\documentclass[11pt]{amsart}
\usepackage{amsfonts}
\usepackage{cite}

\usepackage{amsmath}
\usepackage{amssymb}
\usepackage{latexsym}
\usepackage{amscd}
\usepackage{mathrsfs}
\usepackage[all]{xy}

\newdimen\AAdi%
\newbox\AAbo%
\def\AAk#1#2{\s_etbox\AAbo=\hbox{#2}\AAdi=\wd\AAbo\kern#1\AAdi{}}%
\def\AAr#1#2#3{\s_etbox\AAbo=\hbox{#2}\AAdi=\ht\AAbo\raise#1\AAdi\hbox{#3}}%
\font\tenmsb=msbm10 at 12pt \font\sevenmsb=msbm7 at 8pt
\font\fivemsb=msbm5 at 6pt
\newfam\msbfam
\textfont\msbfam=\tenmsb \scriptfont\msbfam=\sevenmsb
\scriptscriptfont\msbfam=\fivemsb

\newtheorem{theorem}{Theorem}

\newtheorem{remark}[theorem]{Remark}

\newtheorem{lemma}[theorem]{Lemma}
\newtheorem{proposition}[theorem]{Proposition}

\numberwithin{equation}{section} \numberwithin{theorem}{section}

\renewcommand{\topmargin}{0cm}
\renewcommand{\oddsidemargin}{5mm}
\renewcommand{\evensidemargin}{5mm}
\renewcommand{\textwidth}{150mm}
\renewcommand{\textheight}{230mm}

\def\C{\mathbb C}
\def\R{\mathbb R}

\def\S{\mathbb S}

\def\na{\nabla}
\def\bn{\overline\nabla}

\def\ir#1{\mathbb R^{#1}}

\def\f#1#2{\frac{#1}{#2}}

\def\a{\alpha}
\def\be{\beta}

\def\r{\Re_{I\!V}}

\def\p#1{\partial #1}

\def\de{\delta}
\def\De{\Delta}
\def\e{\eta}

\def\ep{\epsilon}

\def\G{\Gamma}
\def\g{\gamma}
\def\k{\kappa}

\def\la{\lambda}
\def\La{\Lambda}
\def\lan{\langle}
\def\ran{\rangle}

\def\Om{\Omega}
\def\th{\theta}
\def\Th{\Theta}

\def\Si{\Sigma}

\def\r{\rho}

\usepackage{color}

\begin{document}

\title
{Existence and non-existence of  minimal graphs}

\author{Qi Ding}
\address{Shanghai Center for Mathematical Sciences, Fudan University, Shanghai 200438, China}
\email{dingqi@fudan.edu.cn}
\author{J. Jost}
\address{Max Planck Institute for Mathematics in the Sciences, Inselstr. 22, 04103 Leipzig, Germany}
\email{jost@mis.mpg.de}
\author{Y.L. Xin}
\address{Institute of Mathematics, Fudan University, Shanghai 200433, China}
\email{ylxin@fudan.edu.cn}

\begin{abstract}
We study the Dirichlet problem for minimal surface systems in
arbitrary dimension and codimension via mean curvature flow,
and obtain the existence of minimal graphs over arbitrary mean convex bounded $C^2$ domains for a large class of prescribed boundary data.
This result can be seen as a natural generalization of the classical sharp criterion for solvability of the minimal surface equation by Jenkins-Serrin.
In contrast, we also construct a class of prescribed boundary data on just mean convex domains for which the Dirichlet problem in codimension 2
is not solvable. Moreover, we study existence and the uniqueness of minimal graphs by perturbation.
\end{abstract}

\thanks{The authors would like to express their sincere
gratitude to the referees for valuable comments that  improved the quality of the manuscript}

\date{}

\maketitle \tableofcontents

\section{Introduction}
The Dirichlet problem for minimal graphs is one of the classical problems in the theory of nonlinear elliptic PDEs. It has been investigated for over a century, starting with the fundamental work of  Bernstein \cite{B}, Haar \cite{H} and Rado \cite{R}, and it has inspired the development of methods for solving   nonlinear elliptic PDEs and the regularity of their solutions. Many deep and important results were achieved, and for instance the papers of Jenkins-Serrin \cite{JS} and Lawson-Osserman \cite{LO} can be considered as classics in the field. Nevertheless, this problem still poses difficult challenges when we move from the classical case of surfaces in $\R^3$ to minimal graphs of arbitrary dimension and codimension. In this paper, we make a systematic new contribution to that general problem.

In codimension 1, the graphic function that describes a  minimal graph over an $n$-dimensional domain $\Om\subset\R^n$ satisfies the minimal surface equation
\begin{equation}\aligned\label{MS0}
\left(1+|Du|^2\right)\De u-\sum_{i,j=1}^n u_iu_ju_{ij}=0\qquad \mathrm{in}\ \Om.
\endaligned
\end{equation}
For $n=2$, when the domain $\Om$ is convex, the Dirichlet problem for
\eqref{MS0} is solvable for arbitrary continuous boundary data. This
was achieved by successive efforts in the papers of  Bernstein
\cite{B}, Haar \cite{H} and Rado \cite{R} already mentioned. On the
other hand, the Dirichlet problem is not necessarily solvable when
$\Om$ is non-convex, as pointed out by Bernstein. In fact, Finn \cite{F}
constructed such counterexamples. For $n>2$, Gilbarg \cite{G} and
Stampacchia \cite{S} established the existence of solutions to
\eqref{MS0} for smooth boundary data and strictly convex smoothly
bounded domains. Jenkins-Serrin \cite{JS} relaxed convexity of the
domain to mean convexity, and  gave a sharp criterion for the
solvability. When $\p\Om$ is not mean convex, they found smooth
boundary data for which the Dirichlet problem of \eqref{MS0} is not
solvable (see also \cite{GT}). When, however, some smallness condition is imposed on the boundary data, depending on the geometry of the boundary, then Jenkins-Serrin \cite{JS} could still solve the Dirichlet problem. This, in fact, also follows from earlier work of Korn \cite{Korn} by interpolation, see \cite{Wid,Wil}. Williams \cite{Wil} could solve the Dirichlet problem for boundary data with small Lipschitz norm. Also, the solution of the Dirichlet problem for the minimal surface equation is unique for  fixed $C^0$-boundary data.   Thus, the situation for codimension 1 can be considered as well understood.

In higher codimensions, however, the situation is much more complicated, as was shown in a seminal paper by Lawson-Osserman \cite{LO}. In dimension $n=2$, they  showed the solvability  of the Dirichlet problem for the minimal surface system with arbitrary continuous boundary data on a bounded convex domain in $\R^2$.   Again, in general, these solutions are not unique. Uniqueness fails also for reasons that do not apply in codimension 1. In fact, Sauvigny \cite{Sau} showed that from non-unique parametric solutions, that is, solutions that cannot be represented as graphs, in dimension 2 and codimension 1, one can obtain non-unique graphic solutions in dimension 2 and codimension $\ge 3$.
Moreover, Xu-Yang-Zhang \cite{XYZ} obtained the existence of boundary functions on unit disks for which infinitely many analytic solutions and at least one nonsmooth Lipschitz solution exist simultaneously.

In dimension $n\ge 4$, Lawson and Osserman  gave non-existence examples for Dirichlet problems on unit balls with higher codimensions (Theorem 6.1, \cite{LO}). It is therefore natural to investigate under which conditions on the  boundary data  the Dirichlet  problem for minimal graphs with dimension $n>2$ and codimension $m\ge 2$ can be solved.

Reflecting upon the fact that for codimension 1, there is a general existence result, while for higher codimension there are severe obstructions to the existence, it is natural to first consider situations that can be seen as some perturbation or extension of the codimension 1 result to higher codimension. That is the approach that we take in this paper.

We study the Dirichlet problem for minimal graphs with arbitrary dimensions and codimensions, and obtain existence and uniqueness of minimal graphs over arbitrary mean convex bounded $C^2$ domain for a large class of prescribed boundary data, see Theorem \ref{main0} below. While on one hand, our result includes the classical result of Jenkins-Serrin in the codimension one case, on the other hand, it provides an alternative condition for the solvability of the Dirichlet problem for minimal graphs over mean convex domains in higher codimensions.

Let $\Om$ be a bounded domain in $\R^n$ with $\p\Om\in C^2$. Let $d(x)=d(x,\p\Om)$ for all $x\in\Om$, and $\la_\Om$ be the maximum of 0 and the largest eigenvalue of $D^2d$ on $\p\Om$ (see section 2 for more details). Note that $\la_\Om=0$ for the convex $\Om$.

\begin{theorem}\label{main0}
For any mean convex bounded $C^2$ domain $\Om$ with diameter $l$, $m\ge2$ and any $\varphi\in C^2(\overline{\Om})$, there is a constant $\ep_{\varphi}>0$ depending on $n,m$, $\la_\Om l$, $\sup_{\Om}|D\varphi|$ and $l\sup_{\Om}|D^2\varphi|$ such that if $\psi=(\psi^1,\dots ,\psi^{m-1}) \in C^2(\overline{\Om},\R^{m-1})$ satisfies
\begin{equation}\label{000}\aligned
\sum_{\a=1}^{m-1}\left(l\sup_{\Om}|D^2\psi^\a|+\sup_{\Om}|D\psi^\a|\right)\le \ep_{\varphi},
\endaligned
\end{equation}
then there is a solution $u=(u^1,\cdots,u^m)\in C^\infty(\Om,\R^m)\cap C^{1,\g}(\overline{\Om},\R^m)$ to the minimal surface system
\begin{equation}\label{000a}
\left\{\begin{split}
g^{ij}u^\a_{ij}=0\qquad \mathrm{in}\ \Om\\
u^\a=\psi^\a\qquad \mathrm{on}\ \p\Om\\
\end{split}\right.\qquad\qquad\qquad \mathrm{for}\ \a=1,\cdots,m,
\end{equation}
with $g_{ij}=\de_{ij}+\sum_\a u^\a_iu^\a_j$ and $\psi^m=\varphi$, where $\g$ is an arbitrary constant in $(0,1)$.
\end{theorem}
The important point here is that the smallness assumption \eqref{000} is only imposed on $m-1$ of the boundary components $\psi^1,\cdots,\psi^{m-1}$. The remaining component
is an arbitrary $C^2$-function. Therefore, as already mentioned, our result includes those obtained earlier for the codimension 1 case.
The constant $\ep_{\varphi}$ can in principle be computed explicitly, but since our value is presumably far from optimal, we do not bother to do so. 
In any case, by the Lawson-Osserman counterexample, some restriction on the boundary data is necessary. 
Moreover,  the constant $\ep_{\varphi}$ should be small, but it cannot be chosen independently of $\varphi$ (see Theorem \ref{NONEXIST} for details).
We should also mention that Wang \cite{W} obtained some results  without uniqueness when all the boundary data are small. But because of that assumption, his results do not include those known for the codimension 1 case.

Theorem \ref{main0} looks like a perturbation  of the codimension 1 case to higher codimension, but it cannot be obtained easily from  the implicit function theorem and geometric measure theory. The reason is that the constant $\ep_{\varphi}$ is independent of the upper bound for the curvature of $\p\Om$.
In general, the elliptic system \eqref{000a} lacks uniqueness, which makes it difficult to study with classic continuity methods. 
Our strategy is utilizing the (graphic) mean curvature flow (with boundary) to approach \eqref{000a}, where the flow is a parabolic version of \eqref{000a}. 
Although this method has been studied in some cases (see \cite{Hu}\cite{W} for instance, and see \cite{Sm} for a survey on the flow of higher codimension),
the difficulty here is to show there are geometric quantities uniformly bounded along the flow under our initial condition that can control the flow.

Let $f=(f^1,\cdots,f^m)$ denote a (short-time) solution of the parabolic system corresponding to \eqref{000a}. Let $v_f$ denote the slope function of graph$_{f(\cdot,t)}$ (see \eqref{DEfvphi}), and $\Th_f$ denote the function related to the 2-dilation of $f$ (see \eqref{2-dil-phi} and \eqref{DEFThf}). For $\Th_f>0$, $\Th_f$  attains its minima and $v_f$ attains its maxima both at the parabolic boundary (see Lemma 5.3 in \cite{TW} and \eqref{evologvf}).
Under the a priori hypothesis  $\Th_f>0$, we derive interior gradient estimates of $f^1,\cdots,f^{m-1}$ using H\"older gradient estimates of the parabolic version and Huisken's monotonicity formula \cite{Hui}, and boundary gradient estimates of $f^1,\cdots,f^{m}$ using suitable auxiliary functions. 
In particular, we further need $|Df^1|,\cdots,|Df^{m-1}|$ small to deduce an 'effective' boundary gradient estimate of $f^m$ (compared with $v_f$). Moreover, the interior gradient estimate of $f^m$ is derived via $v_f$. In all, the estimates of $v_f$ and $\Th_f$ are not independent, but inseparable. After a delicate computation, we can prove that $v_f$ is uniformly bounded and $\Th_f>0$ along the flow under our initial condition. With \cite{DJX}, we are able to deduce the uniform $C^{1,\g}$-estimate for $f$, which implies the long-time existence of the flow. As a result, there is a subsequence $t_i\rightarrow\infty$ so that $\mathrm{graph}_{f(\cdot,t_i)}$ converges to a solution of \eqref{000}.

We can also study the Dirichlet problem of the system \eqref{000a} by perturbation of a given minimal graph of codimension one using the implicit function theorem. However, in this situation we need the bound of the curvature of $\p\Om$ from both above and below (see Theorem \ref{ep-Ex}, compared with Theorem \ref{main0}).
When comparing our result with the Lawson-Osserman counterexample,
we are lead to the question to what extent also non-perturbative
results are possible. But this is a question for future research.

On convex domains, we can control the constants to obtain an
existence result for the Dirichlet problem with a quantitative bound related to 
the counterexample of \cite{LO} and the Bernstein results in \cite{JXY} (see Theorem \ref{mainMCF*} for the proof).
Here, the role of the constant $b_0$ below lies in controlling the gradient of the graphic mean curvature flow with the boundary condition satisfying \eqref{cond0} below.
\begin{theorem}\label{mainMCF*}
For a convex bounded $C^2$ domain $\Om$ and any constant $b_0\in(1,9]$, let $\psi^1,\cdots,\psi^{m}$ be $C^2$-functions on $\overline{\Om}$ with
\begin{equation}\aligned\label{cond0}
\sum_{\a=1}^{m}\left(enlb_0\sup_{\Om}|D^2\psi^\a|+\sup_{\Om}|D\psi^\a|\right)<\sqrt{b_0-1},
\endaligned
\end{equation}
where $l$ is the diameter of $\Om$. Then there is a solution $u=(u^1,\cdots,u^m)\in C^\infty(\Om,\R^m)\cap C^{1,\g}(\overline{\Om},\R^m)$ for any $\g\in(0,1)$ to the minimal surface system \eqref{000a} with $u=\psi$ on $\p\Om$ and $\sup_\Om\left(\mathrm{det}\left(\de_{ij}+u^\a_iu^\a_j\right)\right)^{1/2}<\sqrt{b_0}$.
\end{theorem}

As indicated, the  construction of Lawson-Oseerman gives many non-existence examples
for higher codimensions (Theorem 6.1, \cite{LO}). In section 6, we also construct many non-existence examples for the Dirichlet problem for minimal graphs in dimension $>2$ and codimension $2$ over any mean convex (but not convex) domain, see Theorem \ref{NONEXIST} for concrete results. In particular, this implies that the constant $\ep_{\varphi}$ in Theorem \ref{main0} cannot be only small but independent of $\varphi$.

In the last section we consider the uniqueness of the Dirichlet problem for minimal graphs over a mean convex domain. In \cite{LO} Lawson-Osserman
showed that there exists a real analytic function $\phi:\p D\to \ir{2}$ with the property that there are at least three distinct solutions of the corresponding problem, where $D$ is the  unit disk in $\ir{2}$ (Theorem 5.1, \cite{LO}). Moreover, one of these solutions represents an unstable minimal surface.  Lee-Wang \cite{LW} also considered the uniqueness. They proved a uniqueness theorem for nonparametric minimal submanifolds whose graphic functions are both distance-decreasing and equal on the boundary. In contrast, Sauvigny \cite{Sau} developed a general construction to produce non-unique solutions in codimension $\ge 3$. We prove a new uniqueness result, Proposition \ref{UNI}, which states that the solution in Theorem \ref{ep-Ex} is unique under a certain condition on the higher regularity of the boundary.

\section{Preliminaries}

For an open set $\Om\subset\R^n$, we consider a $C^2$ isometric immersion $X:\ \Om\rightarrow\R^{n+m}$. Then $X$ is a minimal immersion if and only if
\begin{equation}\aligned\label{Oms}
\sum_{i,j=1}^n\f{\p}{\p x_i}\left(\sqrt{\mathrm{det} g_{kl}}g^{ij}\f{\p X}{\p x_j}\right)=0,
\endaligned
\end{equation}
where $g_{ij}=\lan\p X/\p x_i,\p X/\p x_j\ran$, and the matrix $(g^{ij})$ is the inverse of $(g_{ij})$.
The immersion $X$ is called \emph{non-parametric} if it has the form $X(x)=(x,u(x))$ for some vector-valued function $u=(u^1,\cdots,u^m):\ \Om\rightarrow\R^m$.
Putting $U(x)=(U^1(x),\cdots,U^{n+m}(x))=(x,u(x))$, in this case the system \eqref{Oms} becomes
\begin{equation}\label{npms}
\sum_{i,j=1}^n\f{\p}{\p x_i}\left(\sqrt{\mathrm{det} g_{kl}}g^{ij}\f{\p U^a}{\p x_j}\right)=0\qquad \text{for all}\ a,
\end{equation}
where now $g_{ij}=\de_{ij}+\sum_{\a=1}^m\p_{x_i}u^\a\p_{x_j}u^\a$.
From this, one sees \cite{O} (or \cite{LO})  that \eqref{npms} may also be written as
\begin{equation}\aligned\label{Nms}
\sum_{i,j=1}^ng^{ij}\f{\p^2 u^\a}{\p x_i\p x_j}=0\qquad \mathrm{for}\ \a=1,\cdots,m.
\endaligned
\end{equation}
If $m=1$, the above minimal surface system reduces to the following single equation
\begin{equation}\aligned\label{Nms1}
\mathrm{div}\left(\f{Du}{\sqrt{1+|Du|^2}}\right)=0.
\endaligned
\end{equation}
De Giorgi \cite{De} showed that Lipschitz solutions to the minimal surface equation \eqref{Nms1} are smooth (see also \cite{M}, \cite{S}, for instance).
However, such regularity cannot extend to  Lipschitz solutions of the minimal surface system \eqref{Nms} since  Lawson-Osserman \cite{LO} have constructed non-parametric minimal cones, that is, nonsmooth Lipschitz solutions.
Concerning interior regularity, Morrey \cite{Mo1,Mo2} showed that $C^1$ solutions $u$ of the system \eqref{Nms} are smooth.

Let $\De$ and $D$ denote the Laplacian and the Levi-Civita connection of $\R^n$, respectively.
For any Lipschitz function $\varphi$ on $\Om$, let $|D\varphi|$ denote the Lipschitz norm of $\varphi$ at the considered point.
If $\varphi\in C^1(\Om)$, then $|D\varphi|$ is the norm of the gradient of $D\varphi$.
If $\varphi\in C^2(\Om)$, we define $|D^2\varphi|=\sup_{|\xi|=1}|D^2\varphi(\xi,\xi)|$.
For any vector-valued function $\phi=(\phi^1,\cdots,\phi^m)\in C^1(\Om,\R^m)$, we define 2-dilation of $\phi$ on $\Om$ by
\begin{equation}\aligned\label{2-dil-phi}
\sup_\Om\big|\La^2d\phi\big|=\sup_{x\in\Om}\big|\La^2d\phi(x)\big|=\sup_{x\in\Om,1\le i<j\le n}\mu_i(x)\mu_j(x),
\endaligned
\end{equation}
where $\{\mu_k(x)\}_{k=1}^n$ are the singular values of the Jacobi matrix $d\phi(x)=(\p_{x_i}\phi^\a(x))_{n\times m}$.

We denote the distance from $\p \Om$ by
$$d(x)=d(x,\p\Om)=\inf_{y\in\p\Om}|x-y|$$
for each $x\in\overline{\Om}$ and
$$\Om_{s}\triangleq\{y\in\Om|\ d(y,\p\Om)>s\}\qquad \mathrm{for\ any}\ s>0.$$
We further assume $\p\Om\in C^2$.
Let $\la_1(D^2d),\cdots,\la_{n-1}(D^2d),0$ denote the $n$ eigenvalues of $(d_{ij})_{n\times n}$ at the points in $\overline{\Om}$ where $d$ is twice differentiable. (This is the case in some neighborhood of $\p \Om$.)
Then $-\la_1(D^2d),\cdots,-\la_{n-1}(D^2d)$ on $\p\Om$ are the principal curvatures of $\p\Om$.
If $\max_{1\le i\le n-1}\la_i(D^2d)\le0$ on $\p\Om$, then $\Om$ is convex. If $\sum_{i=1}^{n-1}\la_i(D^2d)\le0$ on $\p\Om$, then $\p\Om$ is mean-convex, i.e., the mean curvature of $\p\Om$ is nonnegative. Let $\la_\Om$ be the maximum of zero and the largest eigenvalue of $D^2d$ on $\p\Om$, i.e.,
$$\la_{\Om}\triangleq\max_{\p\Om}\left\{0,\max_{1\le i\le n-1}\la_i(D^2d)\right\}\ge0.$$
For any $x\in\Om$ where  $d$ is differentiable (for instance, in some neighborhood of $\p \Om$),  there exists a unique $y_x\in\p\Om$ such that $d(x)=|x-y_x|$. In particular, $d$ is twice differentiable at $x$.
From Lemma 14.17 in \cite{GT}, we have
\begin{equation}\aligned\label{laOm}
\max_{1\le i\le n}\la_i(D^2d)\le\la_\Om\qquad \mathrm{at}\ x.
\endaligned
\end{equation}
Let $r_{_\Om}$ be the infimum of the radii of exterior balls of $\Om$.
Namely, $r_{_\Om}$ is the largest constant such that for any $p\in\p\Om$ there is a unique open ball $B_{r_{_\Om}}(q)\subset\R^n\setminus\Om$ centered at $q$ and with the radius $r_{_\Om}$ such that $\overline{B_{r_{_\Om}}(q)}\cap\overline{\Om}=p$. It is clear that  $r_{_\Om}\le1/\la_\Om$. Thus, if $\la_\Om=0$, then $1/\la_\Om=\infty$.
In general, however, $r_{_\Om}\neq1/\la_\Om$ when $\la_\Om>0$.

{\bf Notational conventions:} Unless the contrary is explicitly stated, we assume that the considered minimal graphs or mean curvature flows have dimension $n\ge2$. For a vector-valued function $\phi=(\phi^1,\cdots,\phi^m):\Om\to\R^m$, $\phi^\a$ denotes its $\a$-th component, $\phi^\a_i$ denotes $\p_{x_i}\phi^\a$, and $v_\phi$ denotes the slope function of graph$_{\phi}$ defined by
\begin{equation}\aligned\label{DEfvphi}
v_\phi\triangleq\sqrt{\det\left(\de_{ij}+\sum_{\a=1}^m \phi_i^\a \phi_j^\a\right)}.
\endaligned
\end{equation}
For simplicity, we denote $C^k(K,\R^m)$ by $C^k(K)$ for any integer $k\ge0$ and any open (or closed) set $K\subset\R^n$.
The Einstein summation convention over repeated indices will be used.  Greek indices $\a,\be$ take their values in the set $\{1,\cdots,m\}$.

\section{Boundary gradient estimates for the mean curvature flow}

Let $\Om$ be an open set in $\R^n$, and $T$  a positive constant. 
For each $f=(f^1,\cdots,f^m)\in C^2(\Om\times(0,T))$,
let $F_t$ be of the form
$F_t(x_1,\cdots,x_n)=(x_1,\cdots,x_n,f^1(x,t),\cdots,f^m(x,t))$
with $x=(x_1,\cdots,x_n)\in\Om$, $t\in(0,T)$ such that the graph$_{f(\cdot,t)}=\{(x,f(x,t))|\, x\in\Om\}\subset\R^{n+m}$ moves along mean curvature flow, i.e.,
$$\f{d F_t}{dt}=H_t(x),$$
where $H_t$ denotes the mean curvature of the graph$_f$.
Then $f=(f^1,\cdots,f^m)$ satisfies the parabolic equations
\begin{equation}\aligned
\f{\p f^\a}{\p t}=\f{d f^\a}{dt}-\f{\p f^\a}{\p x_j}\f{\p x_j}{\p t}=\f1{\sqrt{\det{g_{kl}}}}\p_i\left(g^{ij}\sqrt{\det{g_{kl}}}f^\a_j\right)-\f{f^\a_j}{\sqrt{\det{g_{kl}}}}\p_i\left(g^{ij}\sqrt{\det{g_{kl}}}\right)=g^{ij}f^\a_{ij}
\endaligned
\end{equation}
on $\Om\times(0,T)$,
where $g_{ij}=\de_{ij}+\sum_\a f^\a_if^\a_j$, and $(g^{ij})$ is the inverse matrix of $(g_{ij})$.

In this section, we will study boundary gradient estimates for the mean curvature flow.
\begin{lemma}\label{BGEMCF}
Let $\Om$ be a bounded domain in $\R^n$ with $\p\Om\in C^2$ and diameter $l$, and $\psi=(\psi^1,\cdots,\psi^m)\in C^2(\overline{\Om})$. Suppose there is a solution $(f^1,\cdots,f^m)\in C^\infty(\Om\times(0,T])\cap C^0(\overline{\Om}\times[0,T])$ to the flow
\begin{equation}\label{MCF3}
\left\{\begin{split}
\f{\p f^\a}{dt}=&g^{ij}f^\a_{ij}\qquad \mathrm{in}\ &\Om\times(0,T]\\
f^\a(\cdot,0)=&\psi^\a \qquad\ \ \mathrm{on}\ &\Om\times\{0\}\\
f^\a(\cdot,t)=&\psi^\a\qquad \mathrm{on}\ &\p\Om\times[0,T]\\
\end{split}\right.\qquad\qquad \mathrm{for}\ \a=1,\cdots,m,
\end{equation}
where $(g^{ij})$ is the inverse matrix of $g_{ij}=\de_{ij}+\sum_\a f^\a_if^\a_j$. Then the following boundary gradient estimate holds:
\begin{equation}\aligned\label{Dfaf1T}
\sup_{\p\Om\times[0,T]}|Df^\a|\le nl\left(1+\la_{f}^2\right)\ e^{1+\f{(n-1)l}{r_{_\Om}}\left(1+\la_{f}^2\right)}\sup_{\Om}|D^2\psi^\a|+\sup_{\Om}|D\psi^\a|,
\endaligned
\end{equation}
where $\la_{f}^2$ is the supremum of the largest eigenvalue of $Df(Df)^T$ on $\Om\times(0,T)$.
\end{lemma}
\begin{proof}
We consider a point $p\in\p\Om$; without loss of generality (after a translation), we can assume $p=0$ and $\overline{B_{r_{_\Om}}}(-r_{_\Om} E_n)\cap\overline{\Om}=0$, where $E_n=(0,0,\cdots,1)\in\R^n$ and $r_{_\Om}E_n=(0,0,\cdots,0,r_{_\Om})$.
We define a function
$$S_\pm^\a(x,t)=\f{\Th_\a}{\th}\left(1-e^{-\th \r(x)}\right)\pm\left(f^\a(x,t)-\psi^\a(x)\right)$$
on $\overline{\Om}\times[0,T]$, where $\r(x)=|x+r_{_\Om}E_n|-r_{_\Om}$, and $\th,\Th_\a$ are positive constants to be defined later.
Obviously, $\r>0$ on $\Om$. Put $y_i=x_i$ for $i=1,\cdots,n-1$ and
$y_n=x_n+r_{_\Om}E_n$. Note that every eigenvalue of $(g^{ij})$ is
between $\left(1+\la_{f}^2\right)^{-1}$ and 1. Since the matrix
  $(\de_{ij}-y_iy_j|x+r_{_\Om}E_n|^{-2})$ has eigenvalues 0 and 1 (of
  multiplicity $(n-1)$), there exist an orthonormal $(n\times n)$-matrix $P=(p_{ij})$ and a diagonal $(n\times n)$-matrix $\La=(\La_{ij})=\mathrm{diag}\{0,1,\cdots,1\}$ so that 
$\de_{ij}-y_iy_j|x+r_{_\Om}E_n|^{-2}=P_{ik}\La_{kl}P_{jl}$. Clearly, the
matrix $(P_{ik}g^{ij}P_{jl})_{i,j=1,\cdots,n}$ is positive definite with
eigenvalues $\le1$, which implies that each element of the matrix $(P_{ik}g^{ij}P_{jl})_{i,j=1,\cdots,n}$ $\le1$. Then 
\begin{equation}\aligned\label{gijdeijyiyj}
g^{ij}\left(\de_{ij}-\f{y_iy_j}{|x+r_{_\Om}E_n|^2}\right)=g^{ij}P_{ik}\La_{kl}P_{jl}
\le n-1.
\endaligned
\end{equation}
For any $\a\in\{1,\cdots,m\}$, with \eqref{gijdeijyiyj} we have
\begin{equation}\aligned
\f{\p S_\pm^\a}{\p t}-g^{ij}\p_{ij}S_\pm^\a=&-g^{ij}\Th_\a e^{-\th \r}\left(\f{\de_{ij}}{|x+r_{_\Om}E_n|}-\f{y_iy_j}{|x+r_{_\Om}E_n|^3}-\f{y_iy_j\th}{|x+r_{_\Om}E_n|^2}\right)
\pm g^{ij}\p_{ij}\psi^\a\\
\ge&\Th_\a e^{-\th \r}\left(\f{\th}{1+\la_{f}^2}-\f{n-1}{\r+r_{_\Om}}\right)-n\sup_{\Om}|D^2\psi^\a|
\endaligned
\end{equation}
on $\Om\times(0,T]$.
Let $l$ be the diameter of $\Om$ and $\th=\f1l+\f{n-1}{r_{_\Om}}\left(1+\la_{f}^2\right)$. If we set
$$\Th_\a=nl\left(1+\la_{f}^2\right)\ e^{1+\f{(n-1)l}{r_{_\Om}}\left(1+\la_{f}^2\right)}\sup_\Om|D^2\psi^\a|,$$
then
\begin{equation}\aligned
\f{\p S_\pm^\a}{\p t}-g^{ij}\p_{ij}S_\pm^\a>0.
\endaligned
\end{equation}
By the maximum principle, it is clear that $S_\pm^\a>0$ on $\Om\times[0,T]$.
Namely,
\begin{equation}\aligned
\left|f^\a(x,t)-\psi^\a(x)\right|<\f{\Th_\a}{\th}\left(1-e^{-\th \r(x)}\right) \qquad \mathrm{for\ any}\ (x,t)\in\Om\times[0,T].
\endaligned
\end{equation}
Hence at the point $p=0$, it follows that
\begin{equation}\aligned
\left|Df^\a(p,t)-D\psi^\a(p)\right|\le\Th_\a \qquad \mathrm{for\ any}\ t\in[0,T].
\endaligned
\end{equation}
As $p$ is an arbitrary point in $\p\Om$,  the proof is complete.
\end{proof}
\textbf{Remark.} We do not use the structure of $g^{ij}$ in the above proof. Though the general boundary gradient estimates are well-known, such as in \cite{Li} for parabolic equations or in \cite{GT} for elliptic equations, we use a different auxiliary function here. In the case of a convex $\Om$, we see that our estimate \eqref{Dfaf1T} is stronger than Theorem 3.1 in \cite{W} if we let $r_{_\Om}\rightarrow\infty$ in \eqref{Dfaf1T}.

Comparing Lemma \ref{BGEMCF}, we have a boundary gradient estimate for the mean curvature flow depending on $\la_\Om$, but independent of $r_{_\Om}$.
\begin{lemma}\label{PDfa}
Let $\Om$ be a bounded domain in $\R^n$ with $\p\Om\in C^2$ and diameter $l$, and $\psi=(\psi^1,\cdots,\psi^m)\in C^2(\overline{\Om})$.
Let $f=(f^1,\cdots,f^m)$ be a smooth solution to \eqref{MCF3} on $\Om\times(0,T]$ with $f\in C^0(\overline{\Om}\times[0,T])$ and boundary data $\psi$. Then
\begin{equation}\aligned
\sup_{\p\Om\times[0,T]}|Df^\a|\le e^{1+2(n-1)\la_\Om l\left(1+\la_{f}^2\right)}\left(nl\left(1+\la_{f}^2\right)\sup_\Om|D^2\psi^\a|+3\sup_\Om|D\psi^\a|\right),
\endaligned
\end{equation}
where $\la_{f}^2$ is the supremum of the largest eigenvalue of $Df(Df)^T$ on $\Om\times(0,T)$.
\end{lemma}
\begin{proof}
We consider a point $p\in\p\Om$; without loss of generality (after a translation), we can assume $p=0$ and $E_n=(0,0,\cdots,1)\in\R^n$ is the unit normal vector (pointing into $\Om$) to $\p\Om$ at $p=0$. Let $\r(x)=\left|x+\f1{2\la_\Om}E_n\right|-\f1{2\la_\Om}$ for each $x\in\R^n$, and
$\Om_*$ be a connected component in $\left\{x\in\Om\big|\, 0<\r(x)<\f1{2\la_\Om}\right\}$ with $0\in\p\Om_*$.
Then $\Om_*$ is a connected open set in $\left\{\left|x+\f1{2\la_\Om}E_n\right|<\f1{\la_\Om}\right\}$.
By the definition of $\la_\Om$, $\p\Om_*\cap\{\r(x)=0\}=\{0\}$.

We define a function
$$S_\pm^\a(x,t)=\f{\Th_\a}{\th}\left(1-e^{-\th \r(x)}\right)\pm\left(f^\a(x,t)-\psi^\a(x)\right)$$
on $\overline{\Om}\times[0,T]$, where $\th=\f1l+2(n-1)\la_\Om\left(1+\la_{f}^2\right)$ and
$$\Th_\a=e^{1+2(n-1)\la_\Om l\left(1+\la_{f}^2\right)}\left(nl\left(1+\la_{f}^2\right)\sup_\Om|D^2\psi^\a|+2\sup_\Om|D\psi^\a|\right).$$
Then at any point $x$ with $\r(x)=\f1{2\la_\Om}$, we get
\begin{equation}\aligned
\f{\Th_\a}{\th}\left(1-e^{-\th \r(x)}\right)\ge2l\sup_\Om|D\psi^\a|\left(1-e^{-(n-1)}\right)\ge l\sup_\Om|D\psi^\a|,
\endaligned
\end{equation}
which implies $S_-^\a(x,t)\ge\sup_\Om|\psi^\a|$, and $S_+^\a(x,t)\le-\sup_\Om|\psi^\a|$.
From the calculation in the proof of Lemma \ref{BGEMCF}, one has
\begin{equation}\aligned
\f{\p S_\pm^\a}{\p t}-g^{ij}\p_{ij}S_\pm^\a>0\qquad \mathrm{on}\ \Om\times(0,T].
\endaligned
\end{equation}
By the maximum principle, we complete the proof.
\end{proof}

Now we assume that the diameter of $\Om$ satisfies $l=1$, and $\p\Om$ has nonnegative mean curvature pointing into $\Om$ in the rest of this section.
From Lemma 14.17 in \cite{GT}, $\p\Om_t$ also has nonnegative mean curvature.
Let $\phi$ be a $C^2$-function on $[0,\infty)$ with
$$\phi'\ge0\qquad \mathrm{and} \qquad\phi''\le0\qquad\quad \mathrm{on}\ [0,\infty).$$
Let $\varphi\in C^2(\overline{\Om})$ and $\tilde{\phi}=\phi\circ d+\varphi$, where $d(x)= d(x,\p\Om)$
for all $x\in\overline{\Om}$. For each $w^\a\in C^1(\Om)$ with $\a=1,\cdots,m-1$,
we define
$$a_{ij}\triangleq\de_{ij}+\sum_{\a=1}^{m-1}w^\a_iw^\a_j+\tilde{\phi}_i\tilde{\phi}_j$$
for $1\le i,j\le n$, and let $(a^{ij})$ be the inverse matrix of $(a_{ij})$.
Assume that the matrix $(D\tilde{\phi},Dw^1,\cdots,Dw^{m-1})$ has singular values $\mu_1,\cdots,\mu_n$ with
\begin{equation}\aligned\label{muimujle1}
|\mu_i\mu_j|\le1\qquad\ \mathrm{for\ any}\ i\neq j.
\endaligned
\end{equation}

\begin{lemma}
Suppose that $\phi$ is the function defined as above, and \eqref{muimujle1} holds. Then
\begin{equation}\aligned\label{aijphiij}
a^{ij}\tilde{\phi}_{ij}\le(n-2)\phi'\la_\Om\left(\f{2}{(\phi')^2}\left(|D\varphi|^2+\f{n-1}{1+\mu_1^2}\right)+\f{1}{1+\mu_1^2}\right)+\f{\phi''}{1+\mu_1^2}+n|D^2\varphi|.
\endaligned
\end{equation}
at all differentiable points of $d$ on $\Om$.
\end{lemma}
\begin{proof}
At any fixed point $p\in\Om$ at which $d$ is differentiable, 
we choose a coordinate system such that
$$a_{ij}=\de_{ij}(1+\mu_i^2),$$
with $\mu_1^2\ge\mu_2^2\ge\cdots\ge\mu_n^2$.
From $a_{ii}=1+\sum_{\a=1}^{m-1}w^\a_iw^\a_i+\tilde{\phi}_i\tilde{\phi}_i=1+\mu_i^2$, it follows that
\begin{equation}\aligned\label{mui2phii2}
\mu_i^2\ge|\tilde{\phi}_i|^2.
\endaligned
\end{equation}
Combining \eqref{muimujle1} and \eqref{mui2phii2}, one has
\begin{equation}\aligned
\sum_{i=2}^n|\tilde{\phi}_i|^2\le\sum_{i=2}^n\mu_i^2\le\f{n-1}{\mu_1^2}.
\endaligned
\end{equation}
With the Cauchy-Schwarz inequality we get
\begin{equation}\aligned
(\phi')^2\sum_{i=2}^nd_i^2\le-\sum_{i=2}^n\left(2\phi'd_i\varphi_i+\varphi_i^2\right)+\f{n-1}{\mu_1^2}\le\f12(\phi')^2\sum_{i=2}^nd_i^2+\sum_{i=2}^n\varphi_i^2+\f{n-1}{\mu_1^2}.
\endaligned
\end{equation}
where $d_i=\f{\p d}{\p x_i}$. Thus
\begin{equation}\aligned\label{1-d12}
1-d_1^2=\sum_{i=2}^nd_i^2\le\f2{(\phi')^2}\left(|D\varphi|^2+\f{n-1}{\mu_1^2}\right).
\endaligned
\end{equation}

Recall $d(x)=d(x,\p\Om)$.
In a neighborhood of the point $p$, we choose an orthonormal basis $\{\p_\r\}\cup\{e_i\}_{i=1,\cdots,n-1}$, such that $\p_\r d=1$, $\{e_i\}_{i=1,\cdots,n-1}$ is normal at $p$, and
\begin{equation}\aligned
\f{\p}{\p x_1}=d_1\p_\r+\sqrt{1-d_1^2}e_1.
\endaligned
\end{equation}
Here, 'normal' means $(D_{e_i}e_i)^T=0$ at $p$, where $(\cdot)^T$ denotes the projection onto the tangent bundle of $\p\Om_{d(p)}$.
Since the function $d$ is a constant on $\p\Om_{d(p)}$, then we get
\begin{equation}\aligned\label{Deiei0}
\left(D_{e_i}D_{e_i}-(D_{e_i}e_i)^T\right)d=D_{e_i}D_{e_i}d=0
\endaligned
\end{equation}
for each $i=1,\cdots,n-1$ at $p$.
Since $D_{\p_\r}d=1$ and $\left(D_{\p_\r}D_{\p_\r}-D_{\p_\r}\p_\r\right)d=0$ at $p$, combining \eqref{Deiei0} one has
\begin{equation}\aligned\label{Ded11}
d_{11}=\mathrm{Hess}_d\left(\f{\p}{\p x_1},\f{\p}{\p x_1}\right)=(1-d_1^2)\mathrm{Hess}_d\left(e_1,e_1\right)=-(1-d_1^2)\left(D_{e_1}e_1\right)d,
\endaligned
\end{equation}
and
\begin{equation}\aligned
\De d=\sum_{i=1}^{n-1}\left(D_{e_i}D_{e_i}-D_{e_i}e_i\right)d=-\sum_{i=1}^{n-1}\left\lan D_{e_i}e_i,\p_\r\right\ran \p_\r d=-H_{\p\Om_{d(p)}}\le0,
\endaligned
\end{equation}
where $H_{\p\Om_{d(p)}}$ denotes the mean curvature of $\p\Om_{d(p)}=\{x\in\Om|\, d(x)=d(p)\}$. Since $\p\Om_{d(p)}$ is mean convex, then $\De d\le0$ at $p$.
From \eqref{laOm}, it follows that $-\left(D_{e_i}e_i\right)d\le \la_\Om$ for each $i=1,\cdots,n-1$ at $p$. Then for any $t\in[0,1]$
\begin{equation}\aligned\label{11Ded}
-td_{11}+\De d=&t(1-d_1^2)\left(D_{e_1}e_1\right)d+\De d\le t(1-d_1^2)\left(D_{e_1}e_1\right)d+t(1-d_1^2)\De d\\
=&-t(1-d_1^2)\sum_{i=2}^{n-1}\left(D_{e_i}e_i\right)d\le (n-2)t(1-d_1^2)\la_\Om.
\endaligned
\end{equation}

We shall now compute $a^{ij}\tilde{\phi}_{ij}$ at $p$. From $\mu_1^2\ge\mu_2^2\ge\cdots\ge\mu_n^2$, \eqref{laOm}, \eqref{muimujle1}, \eqref{11Ded}, one has
\begin{equation}\aligned\label{paijdij}
a^{ij}d_{ij}=&\sum_{i=1}^n\f{1}{1+\mu_i^2}d_{ii}
=\f{d_{11}}{1+\mu_1^2}-\f{d_{11}}{1+\mu_2^2}+\f{\De d}{1+\mu_2^2}+\sum_{i=3}^n\left(\f{1}{1+\mu_i^2}-\f{1}{1+\mu_2^2}\right)d_{ii}\\
\le&\f{1}{1+\mu_2^2}\left(-\f{\mu_1^2-\mu_2^2}{1+\mu_1^2}d_{11}+\De d\right)+\sum_{i=3}^n\left(1-\f{1}{1+\mu_2^2}\right)\la_\Om\\
\le&\f{(n-2)\mu_1^2}{1+\mu_1^2}(1-d_1^2)\la_\Om+\f{n-2}{1+\mu_1^2}\la_\Om.
\endaligned
\end{equation}
Noting $\phi''\le0$ and the definition of $|D^2\varphi|$,  we have
\begin{equation}\aligned\label{aijpdidj}
a^{ij}\left(\phi''d_id_j+\varphi_{ij}\right)\le&\phi''\sum_{i=1}^n\f{d_i^2}{1+\mu_i^2}+n|D^2\varphi|\le\phi''\f{\sum_{i=1}^nd_i^2}{1+\mu_1^2}+n|D^2\varphi|=\f{\phi''}{1+\mu_1^2}+n|D^2\varphi|.
\endaligned
\end{equation}
Combining \eqref{1-d12}\eqref{paijdij} and \eqref{aijpdidj}, we obtain
\begin{equation}\aligned
a^{ij}\tilde{\phi}_{ij}=&a^{ij}\left(\phi'd_{ij}+\phi''d_id_j+\varphi_{ij}\right)\\
\le&(n-2)\phi'\la_\Om\left(\f{2}{(\phi')^2}\left(|D\varphi|^2+\f{n-1}{1+\mu_1^2}\right)+\f{1}{1+\mu_1^2}\right)+\f{\phi''}{1+\mu_1^2}+n|D^2\varphi|.
\endaligned
\end{equation}
This completes the proof.
\end{proof}

Denote $|D\varphi|_\Om\triangleq\sup_{x\in\Om}|D\varphi|$ and $|D^2\varphi|_\Om\triangleq\sup_{x\in\Om}|D^2\varphi|$.
Let us deduce another boundary gradient estimate using the structure of $g^{ij}$.
\begin{theorem}\label{DfMCF}
Let $\Om$ be a mean convex bounded domain in $\R^n$ with diameter $l=1$ and $\p\Om\in C^2$.
Let $\psi=(\psi^1,\cdots,\psi^m)\in C^2(\overline{\Om})$ and $f=(f^1,\cdots,f^m)$ be a smooth solution to \eqref{MCF3} on $\Om\times(0,T]$ with $f\in C^0(\overline{\Om}\times[0,T])$ and boundary data $\psi$. Denote $\varphi=\psi^m$ and $\nu=16n(|D^2\varphi|_\Om+1)$.
Let $\k$ be the constant defined by
\begin{equation}\aligned\label{defd0}
\k=\max\left\{64(n-2)\la_\Om(1+|D\varphi|_\Om^2)e^{3|D\varphi|_\Om\nu},2\nu(\sqrt{n}+|D\varphi|_\Om)e^{|D\varphi|_\Om\nu}\right\}.
\endaligned
\end{equation}
If $\sup_{\Om\times(0,T)}\big|\bigwedge^2df\big|\le1$, and $|Df^\a|\le\f{1}{m-1}$ on $\Om_{1/\k}\times[0,T]$ for $\a=1,\cdots,m-1$, and
$$\sup_{\Om\times(0,T)}\mathrm{det}g_{ij}\le\f{2\k^2}{\nu^2},$$
then we have
\begin{equation}\aligned\label{DfmMCF}
\sup_{(x,t)\in\p\Om\times[0,T]}|Df^m(x,t)|\le\f{\k}{\nu}+|D\varphi|_\Om.
\endaligned
\end{equation}
\end{theorem}
\begin{proof}
By the maximum principle for parabolic equations,
\begin{equation}\aligned
\inf_{y\in\Om}\varphi(y)\le f^m(x,t)\le\sup_{y\in\Om}\varphi(y)\qquad \mathrm{for\ all}\ x\in\Om\times[0,T].
\endaligned
\end{equation}
Set
$$\phi(d)=\f1\nu\log\left(1+\k d\right)\qquad \mathrm{on}\ \Om,$$
then
\begin{equation}\aligned
\phi'=\f{\k}{\nu(1+\k d)}>0\qquad \mathrm{and}\qquad \phi''=-\f{\k^2}{\nu(1+\k d)^2}<0.
\endaligned
\end{equation}
Set
$$\tilde{\phi}=\phi\circ d+\varphi\qquad \mathrm{on}\quad \Om.$$
Denote $g_{ij}=\de_{ij}+\sum_{\a=1}^m u^\a_iu^\a_j$.
We claim
\begin{equation}\aligned\label{aijphiij0}
g^{ij}\tilde{\phi}_{ij}=g^{ij}\left(\phi'd_{ij}+\phi''d_id_j+\varphi_{ij}\right)<0
\endaligned
\end{equation}
at each considered point $(q,t)\in\Om\times(0,T]$ with $D\tilde{\phi}(q)=Df^m(q,t)$ and $f^m(q,t)>\tilde{\phi}(q)$.
By the maximum principle for \eqref{aijphiij0} and $\f{\p f^m}{\p t}-g^{ij}f^m_{ij}=0$, we obtain
$$\phi(d(x))+\varphi(x)\ge f^m(x,t)\qquad \mathrm{for}\ \mathrm{any}\ x\in\Om\times[0,T].$$
Analogously to the above argument, one has
\begin{equation}\aligned\nonumber
\phi(d(x))-\varphi(x)\ge -f^m(x,t)\qquad \mathrm{for}\ \mathrm{any}\ x\in\Om\times[0,T].
\endaligned
\end{equation}
Therefore, for any $(x,t)\in\p\Om\times[0,T]$ it follows that
\begin{equation}\aligned
|Df^m(x,t)|\le\sup_{\p\Om}|D(\phi\circ d)|+|D\varphi|_\Om\le\f{\k}{\nu }+|D\varphi|_\Om.
\endaligned
\end{equation}

For completing the proof of this theorem, we only need to show the claim \eqref{aijphiij0} at the point $q$ with $D\tilde{\phi}(q)=Df^m(q,t)$ and $f^m(q,t)>\tilde{\phi}(q)$.
Denote
$$d_0=\f1\k e^{|D\varphi|_\Om\nu}.$$
From $\phi(d_0)>|D\varphi|_\Om\ge\sup_{x\in\Om}\varphi(x)-\inf_{x\in\Om}\varphi(x)$,
there holds $q\in\Om\setminus\overline{\Om_{d_0}}$.
Let $\la_1,\cdots,\la_n$ be the singular values of $f^\a_i$ at $(q,t)$ such that $\la_1^2\ge\la_2^2\ge\cdots\ge\la_n^2$.
Since $0<d(q)<d_0$, then by the definition of $\k$ we have
$$\phi'=\f{\k}{\nu(1+\k d)}\ge\f1{2d_0\nu}\ge \sqrt{n}+|D\varphi|_\Om.$$
Combining the definition of $\k, d_0$ and $1+\la_1^2\le \mathrm{det}g_{ij}\le\f{2\k^2}{\nu^2}$, first we have
\begin{equation}\aligned\label{3.15}
&(n-2)\la_\Om\left(\f{2}{(\phi')^2}\left(|D\varphi|^2+\f{n-1}{1+\la_1^2}\right)+\f{1}{1+\la_1^2}\right)-\f{\phi''}{2\phi'}\cdot\f{1}{1+\la_1^2}\\
\le&(n-2)\la_\Om\left(\f{2|D\varphi|^2}{(\phi')^2}+\f{2(n-1)}{n(1+\la_1^2)}+\f{1}{1+\la_1^2}\right)-\f{\k}{2(1+\k d)}\cdot\f{1}{1+\la_1^2}\\
\le&(n-2)\la_\Om\left(8|D\varphi|^2d_0^2\nu^2+\f{3}{1+\la_1^2}\right)-\f{1}{4d_0}\cdot\f{1}{1+\la_1^2}\\
\le&8(n-2)\la_\Om|D\varphi|^2\f{\nu^2}{\k^2}e^{2|D\varphi|_\Om\nu}-\left(\f1{4d_0}-3(n-2)\la_\Om\right)\f{\nu^2}{2\k^2}\\
=&\f{\nu^2}{8\k^2}\left(-\f1{d_0}+64(n-2)\la_\Om|D\varphi|^2e^{2|D\varphi|_\Om\nu}+12(n-2)\la_\Om\right)\le0.
\endaligned
\end{equation}
Next, let us estimate the remaining terms in \eqref{aijphiij}.
\begin{itemize}
  \item Case 1: $q\in\Om\setminus\Om_{1/\k}$.
Since $\k d(q)\le1$ and $1+\la_1^2\le\f{2\k^2}{\nu^2}$, it follows that
\begin{equation}\aligned\label{3.16}
\f12\f{\phi''}{1+\la_1^2}+n|D^2\varphi|_\Om\le-\f{\k^2}{2\nu(1+\k d)^2}\cdot\f{\nu^2}{2\k^2}+|D^2\varphi|_\Om\le-\f{\nu}{16}+n|D^2\varphi|_\Om\le0;
\endaligned
\end{equation}
  \item Case 2: $q\in\Om_{1/\k}$.
From the assumption in this theorem, one has
\begin{equation}\aligned
1+\la_1^2\le& 1+\left(\sum_{\a=1}^m |Df^\a|\right)^2\le1+\left(|Df^m|+(m-1)\f{1}{m-1}\right)^2\\
\le&1+\left(|D\tilde{\phi}|+1\right)^2\le1+\left(\phi'+|D\varphi|+1\right)^2<5(\phi')^2.
\endaligned
\end{equation}
Then
\begin{equation}\aligned\label{3.16'}
\f12\f{\phi''}{1+\la_1^2}+n|D^2\varphi|_\Om<-\f{\phi''}{10(\phi')^2}+n|D^2\varphi|_\Om=-\f{\nu}{10}+n|D^2\varphi|_\Om\le0.
\endaligned
\end{equation}
\end{itemize}

Combining Lemma \ref{aijphiij} and \eqref{3.15}\eqref{3.16}\eqref{3.16'},
the claim \eqref{aijphiij0} is true. We complete the proof.
\end{proof}

\section{The Dirichlet problem on mean convex domains}

Let $\la_1,\la_2,\cdots,\la_{n}$ be the singular values of a matrix $\{u^\a_i\}_{1\le i\le n,\ 1\le \a\le m}$ ($\la_j=0$ if $\min\{m,n\}<j\le n$).
Let us first prove an algebraic lemma that will be needed in the sequel.
\begin{lemma}\label{Du1a}
If $|Du^1|^2\cdot\sum_{\a=2}^m|Du^\a|^2\le K^2$ for some constant $K>0$, and $\sum_{\a=2}^m|Du^\a|^2\le K$, then $|\la_i\la_j|\le 2K$ for all $i\neq j$.
\end{lemma}
\begin{proof}
By scaling, we only need prove the lemma for $K=1$. For any considered point $x$, we choose an orthonormal coordinate system in its neighborhood such that $D_1u^1=|Du^1|$ at $x$. Now we assume that $D_1u^1=t$ for some constant $t\ge1$. Then by the assumption of the lemma, it follows that
$$\sum_{\a=2}^m|Du^\a|^2\le\f1{t^2}.$$
For any $\xi=(\xi_1,\cdots,\xi_n)\in\S^n$, we have
\begin{equation}\aligned\label{4.1}
\sum_{\a,i,j}u^\a_iu^\a_j\xi_i\xi_j=u^1_1u^1_1\xi_1\xi_1+\sum_{\a\ge2,i,j}u^\a_iu^\a_j\xi_i\xi_j\le t^2\xi_1^2+\f{1}{t^2}.
\endaligned
\end{equation}
By a rearrangement, we can assume that $\la_1^2$ is the maximal eigenvalue of the matrix $\left(\sum_\a u^\a_iu^\a_j\right)_{n\times n}$ with the corresponding eigenfunction $\xi=(\xi_1,\cdots,\xi_n)\in\S^n$. Then from \eqref{4.1} it follows that
\begin{equation}\aligned\nonumber
t^2\le\la_1^2\le t^2\xi_1^2+\f{1}{t^2},
\endaligned
\end{equation}
which implies
$$\xi_1^2\ge1-\f{1}{t^4}.$$
For any $\e=(\e_1,\cdots,\e_n)\in\S^n$ with $\xi\bot\e=0$, one has
\begin{equation}\aligned
\xi_1^2\e_1^2=\left(\sum_{i\ge2}\xi_i\e_i\right)^2\le\sum_{i\ge2}\xi_i^2\sum_{i\ge2}\e_i^2=(1-\xi_1^2)(1-\e_1^2),
\endaligned
\end{equation}
which implies
\begin{equation}\aligned
\e_1^2\le1-\xi_1^2\le\f{1}{t^4}.
\endaligned
\end{equation}
If $\e=(\e_1,\cdots,\e_n)\in\S^n$ is the eigenfunction of $\left(\sum_\a u^\a_iu^\a_j\right)_{n\times n}$ with respect to the second eigenvalue $\la_2^2$, then combining $D_1u^1=|Du^1|=t$ and $\sum_{\a=2}^m|Du^\a|^2\le\f1{t^2}$ one has
\begin{equation}\aligned
\la_2^2=\sum_{\a,i,j}u^\a_iu^\a_j\e_i\e_j= t^2\e_1^2+\sum_{\a\ge2,i,j}u^\a_iu^\a_j\e_i\e_j\le t^2\e_1^2+\f{1}{t^2}\le\f{2}{t^2}.
\endaligned
\end{equation}
So we obtain
\begin{equation}\aligned
\la_1^2\la_2^2\le\left(t^2+\f{1}{t^2}\right)\f{2}{t^2}=2+\f{2}{t^4}\le4.
\endaligned
\end{equation}

If $D_1u^1\le1$, then $\la_1^2\le 2$ clearly. In all, we always have
$$|\la_1\la_2|\le 2.$$
Hence, this completes the proof of Lemma \ref{Du1a}.
\end{proof}

Now let us deduce an interior gradient estimate for the mean curvature flow.
\begin{lemma}\label{KDfbound}
Let $\Om$ be a bounded domain in $\R^n$ with $\p\Om\in C^2$ and diameter $l=1$.
Let $\psi=(\psi^1,\cdots,\psi^m)\in C^2(\overline{\Om})$ and $f=(f^1,\cdots,f^m)$ be a smooth solution to \eqref{MCF3} on $\Om\times(0,T]$ with $f\in C^0(\overline{\Om}\times[0,T])$ and boundary data $\psi$ such that $\big|\La^2df\big|\le1-\ep$ on $\Om\times[0,T]$ for some constant $\ep\in(0,1)$.
Let $\k$ be the constant in Theorem \ref{DfMCF}.
If $\sup_{\Om\times[0,T]}|Df|\le\La$ for some constant $\La>0$, then for any $0<s\le1/\k$ there exists a constant $C_{s,\ep,\La,\psi}$ depending only on $n,m,s,\ep,\La$,  $|D\psi|_{\Om}$ and $|D^2\psi|_{\Om}$ (but independent of $T$) such that
\begin{equation}\aligned\label{DfaC0***}
|Df^\a|_{\Om_s\times[0,T]}\le \f1m+C_{s,\ep,\La,\psi}|D\psi^\a|_{\Om}\qquad \mathrm{for\ every}\ \a=1,\cdots,m.
\endaligned
\end{equation}
\end{lemma}
\begin{proof}
For a point $\mathbf{x}=(x,t)\in\R^n\times\R=\R^{n+1}$, we set $|\mathbf{x}|=\max\{|x|,|t|^{1/2}\}$ and the cylinder
$$Q_R(\mathbf{x})=\left\{\mathbf{y}=(y,\tau)\in\R^{n+1}|\ |\mathbf{x}-\mathbf{y}|<R,\ \tau<t\right\}.$$
From Lemma 3.1 and Lemma 3.2 in \cite{DJX}, there is a general constant $C_{\ep,\La}$ depending only on $n,m,\ep,\La$ such that for any $Q_r(\mathbf{x_0})\subset\Om\times(0,T)$ with $\mathbf{x_0}=(x_0,t_0)$
\begin{equation}\aligned\label{c1gtft}
\sup_{Q_r(\mathbf{x_0})}(r|D^2f|+r^2|\p_tDf|)\le C_{\ep,\La},
\endaligned
\end{equation}
and
\begin{equation}\label{Qr2x0}
\sup_{Q_{r/2}(\mathbf{x_0})}|Df-\xi|\le C_{\ep,\La}\left(r^{-1}\sup_{\mathbf{x}\in Q_r(\mathbf{x_0})}|f(\mathbf{x})-\xi\cdot(x-x_0)-\iota|+r\right)
\end{equation}
for any $\xi\in\R^n\times\R^m$ and $\iota\in\R^m$.
Let $s$ be a constant in $(0,1/\k]$ with $\k$ defined in \eqref{defd0}.
By translation, we assume $0\in\Om_s$. For any $R>0$, denote
$$Q'_R=\left\{\mathbf{y}=(y,\tau)\in\R^{n+1}|\ |y|<R,\ 0<\tau<R^2\right\}.$$
For the fixed $\a\in\{1,\cdots,m\}$, we define
$w(\mathbf{x})=f^\a(\mathbf{x})-D\psi^\a\cdot x-\psi^\a(0)$
for all $\mathbf{x}=(x,t)\in\Om_s\times(0,T)$. From Lemma 12.6 in \cite{Li}, for all $0<r<R$ and $Q'_R\subset\Om\times(0,T)$
\begin{equation}\aligned
\sup_{Q'_r}|w|\le C_{\ep,\La}\left(|D^2\psi^\a|_{\Om}+R^{-2}\sup_{Q'_R}|w|\right)r^2.
\endaligned
\end{equation}
Denote $\mathbf{y_0}=(0,r)\in\Om_s\times\R$. Combining \eqref{Qr2x0}, we have
\begin{equation}
\sup_{Q_{r/2}(\mathbf{y_0})}|Df-D\psi(0)|\le C_{\ep,\La}r\left(1+|D^2\psi|_{\Om}+R^{-1}\sup_{Q'_R}|D\psi|+R^{-2}\sup_{Q'_R}|\psi|\right).
\end{equation}
Hence, there is a general constant $C_{s,\ep,\La,\psi}\ge1$ depending only on $n,m,s,\ep,\La$,$|D\psi|_{\Om}$ and $|D^2\psi|_{\Om}$ such that
for any $x\in\Om_{s/2}$ and $\mathbf{y}=(y,\tau)\in\Om_{s/2}\times(0,T)$ we have
\begin{equation}\aligned\label{Dfytx0}
|Df(y,\tau)-Df(x,0)|\le C_{s,\ep,\La,\psi}\max\{|x-y|,\sqrt{\tau}\}.
\endaligned
\end{equation}
Let $\de$ be a positive constant satisfying $C_{s,\ep,\La,\psi}\sqrt{\de}=\f1m$. If $T\le\de$, then \eqref{Dfytx0} implies \eqref{DfaC0***}.

Now we assume $T>\de$.
Let $\mathbf{x}=(x,t), \mathbf{y}=(y,\tau)\in\Om_{s/2}\times(0,T)$, and we denote $\mathbf{y_x}=(y,t)\in\Om_{s/2}\times(0,T)$.
Note $l=1$, and the definition of $\k$ in \eqref{defd0}.
For $|\mathbf{x}-\mathbf{y}|\le\min\{t,\tau\}$, from \eqref{c1gtft} we have
\begin{equation}\aligned\label{Dfxtytau1}
&|Df(\mathbf{x})-Df(\mathbf{y})|\le |Df(\mathbf{x})-Df(\mathbf{y_x})|+|Df(\mathbf{y})-Df(\mathbf{y_x})|\\
\le& \f{C_{s,\ep,\La,\psi}}{\min\{\sqrt{t},\sqrt{\tau}\}}|x-y|+\f{C_{s,\ep,\La,\psi}}{\min\{t,\tau\}}|t-\tau|\le C_{s,\ep,\La,\psi}|\mathbf{x}-\mathbf{y}|^{1/2}.
\endaligned
\end{equation}
For $|\mathbf{x}-\mathbf{y}|>\min\{t,\tau\}$, from \eqref{Dfytx0} we have
\begin{equation}\aligned\label{Dfxtytau2}
&|Df(\mathbf{x})-Df(\mathbf{y})|\le |Df(\mathbf{x})-Df(x,0)|+|Df(\mathbf{y})-Df(x,0)|\\
\le& C_{s,\ep,\La,\psi}\left(\sqrt{t}+\min\{|x-y|,\sqrt{\tau}\}\right)\le C_{s,\ep,\La,\psi}|\mathbf{x}-\mathbf{y}|^{\f12}.
\endaligned
\end{equation}

Let $\e$ be a Lipschitz function with support in $\Om_{s/4}$ such that $\e=1$ on $\Om_{s/2}$, $|D\e|\le\f cs$ and $|D^2\e|\le\f c{s^2}$ for some absolute constant $c\ge1$.
Let $M_t$ be the graph of $f(\cdot,t)$ in $\R^{n+m}$.
We will see $\e$ and $f^\a(\cdot,t)$ as the functions on $M_t$ by identifying $\e(x,f(x,t))=\e(x)$ and $f^\a(x,f(x,t))=f^\a(x,t)$.
Since $t\in(0,t)\mapsto M_t$ is a mean curvature flow, then
\begin{equation}\aligned\label{ftDeMt=0}
\f{d f^\a}{d t}-\De_{M_t} f^\a=0,
\endaligned
\end{equation}
where $\De_{M_t}$ is the Laplacian of $M_t$. For simplicity, let $\na$, $\overline{D}$ denote Levi-Civita connections of $M_t$ and $\R^{n+m}$, respectively.
Let $e_1,\cdots,e_n$ be a local orthonormal tangent frame of $M_t$ at any considered point.
Then from the definition of $\e$ and \eqref{ftDeMt=0},
\begin{equation}\aligned
\left(\f{d}{d t}-\De_{M_t}\right)\e^2=&2\e\overline{D}\e\cdot H_{M_t}-2|\na\e|^2-2\e\De_{M_t}\e\\
=&-2|\na\e|^2-2\e\sum_i\overline{D}^2\e(e_i,e_i)\le\f{2c^2(n+1)}{s^2}.
\endaligned
\end{equation}
So we have
\begin{equation}\aligned\label{fdtfa2e2}
&\left(\f{d}{d t}-\De_{M_t}\right)\left((f^\a)^2\e^2\right)\\
=&(f^\a)^2\left(\f{d}{d t}-\De_{M_t}\right)\e^2+\e^2\left(\f{d}{d t}-\De_{M_t}\right)(f^\a)^2-2\na (f^\a)^2\cdot\na\e^2\\
\le&\f{2c^2(n+1)}{s^2}(f^\a)^2-2\e^2|\na f^\a|^2+\e^2|\na f^\a|^2+16(f^\a)^2|\na\e|^2\\
\le&\f{2c^2(n+9)}{s^2}(f^\a)^2-\e^2|\na f^\a|^2.
\endaligned
\end{equation}
Let $\Phi(X,t)=(-4\pi t)^{-\f n2}e^{|X|^2/t}$ for all $X\in\R^{n+m}$, $t<0$, and $\Phi_{X_0,t_0}(X,t)=\Phi(X-X_0,t-t_0)$ for all $t_0>0$, $X_0\in\R^{n+m}$.
Combining Huisken's monotonicity formula \cite{Hui} (see also (7) in \cite{EH}, or (1.2) in \cite{CM1}) and \eqref{fdtfa2e2}, we have
\begin{equation}\aligned
\f{d}{dt}\int_{M_t}(f^\a)^2\e^2\Phi_{X_0,t_0}\le&\int_{M_t}\left(\f{d}{d t}-\De_{M_t}\right)\left((f^\a)^2\e^2\right)\Phi_{X_0,t_0}\\
\le&\int_{M_t}\left(\f{2c^2(n+9)}{s^2}(f^\a)^2-\e^2|\na f^\a|^2\right)\Phi_{X_0,t_0}.
\endaligned
\end{equation}
Integrating the above inequality on $[t_1,t_2]\subset[0,\min\{t_0,T\}]$ implies
\begin{equation}\aligned
\int_{t_1}^{t_2}\int_{M_t}\e^2|\na f^\a|^2\Phi_{X_0,t_0}\le\int_{M_t}(f^\a)^2\e^2\Phi_{X_0,t_0}\bigg|^{t_1}_{t_2}+\f{2c^2(n+9)}{s^2}\int^{t_2}_{t_1}\int_{M_t}(f^\a)^2\Phi_{X_0,t_0}.
\endaligned
\end{equation}
Note that $\sup_{\Om\times[0,T]}|f^\a|\le\sup_\Om|\psi^\a|$, and $\sup_{\Om\times[0,T]}|Df|\le\La$. Then by choosing suitable $X_0,t_0$, for any $0<s\le1/\k$ there holds
\begin{equation}\aligned\label{t1t2Dfa2}
\int_{t_1}^{t_2}\int_{\Om_{s/2}}|Df^\a|^2\le C_{s,\ep,\La}(t_2-t_1+1)|\psi^\a|^2_{\Om}\qquad \mathrm{for\ any}\ \a=1,\cdots,m,
\endaligned
\end{equation}
where $C_{s,\ep,\La}>0$ is a general constant depending only on $n,m,s,\ep,\La$.
Let $\omega_n$ be the volume of the unit ball in $\R^n$.
For any $(x,t)\in\Om_s\times[\de,T]$ and $0<s_*<\min\{s/2,\sqrt{\de}\}$, from \eqref{Dfxtytau1}\eqref{Dfxtytau2}\eqref{t1t2Dfa2} we get
\begin{equation}\aligned
\omega_ns_*^{n+2}|f_i^\a(x,t)|\le&\int_{t-s_*^2}^t\int_{B_{s_*}(x)}|f_i^\a(y,\tau)-f_i^\a(x,t)|dyd\tau+\int_{t-s_*^2}^t\int_{B_{s_*}(x)}|f_i^\a(y,\tau)|dy\\
\le& \omega_ns_*^{n+2}C_{s,\ep,\La,\psi}s_*^{1/2}+\left(\omega_ns_*^{n+2}\int_{t-s_*^2}^t\int_{B_{s_*}(x)}|f_i^\a|^2\right)^{1/2}\\
\le& \omega_ns_*^{n+\f52}C_{s,\ep,\La,\psi}+\sqrt{C_{s,\ep,\La}\omega_ns_*^{n+2}(s_*^2+1)}|\psi^\a|_{\Om}.
\endaligned
\end{equation}
Let $\sqrt{s_*}=\f1{mC_{s,\ep,\La,\psi}}$, we have
\begin{equation}\aligned
|f_i^\a(x,t)|\le \f1m+C_{s,\ep,\La}s_*^{-\f{n+2}2}|\psi^\a|_{\Om}.
\endaligned
\end{equation}
For any $(x,t)\in\Om_s\times[0,\de]$, using \eqref{Dfytx0} we can also get \eqref{DfaC0***}.
This completes the proof.
\end{proof}

With the mean curvature flow, we can now prove Theorem \ref{main0}.
\begin{proof}
By scaling, we may assume that the diameter $l=1$.
According to Theorem 8.2 in \cite{Li} (see also Lemma 5.1 in \cite{DJX} for instance), there is a constant $T>0$, $\g\in(0,1)$, a solution $(f^1,\cdots,f^m)\in C^\infty(\Om\times(0,T])\cap C^{1,\g}(\overline{\Om}\times[0,T])$ to the mean curvature flow
\begin{equation}\label{MCFmainMCF}
\left\{\begin{split}
\f{\p f^\a}{dt}=&g^{ij}f^\a_{ij}\qquad \mathrm{in}\ &\Om\times(0,T]\\
f^\a(\cdot,0)=&\psi^\a \qquad\ \ \mathrm{on}\ &\Om\times\{0\}\\
f^\a(\cdot,t)=&\psi^\a\qquad \mathrm{on}\ &\p\Om\times[0,T]\\
\end{split}\right.\qquad\qquad \mathrm{for}\ \a=1,\cdots,m,
\end{equation}
such that $|f^\a|_{1+\g;\overline{\Om}\times[0,T]}\le C_{\Om,\psi}$ for each $\a=1,\cdots,m$, where $C_{\Om,\psi}$ is a constant depending only on $m,n,|D\psi|_\Om,|D^2\psi|_{\Om}$ and the curvature of $\p\Om$.
Here, $|\cdot|_{1+\g;\cdot}$ denotes the (higher order) H\"older norm in the parabolic case (see chapter IV in \cite{Li} or \cite{DJX}).
Let $\la_1,\cdots,\la_n$ be the singular values of the matrix $(f^\a_i)$ at each point in $\overline{\Om}\times[0,T]$ with $\la_1\ge\cdots\ge\la_n\ge0$. Denote
\begin{equation}\aligned\label{DEFThf}
\Theta_f=\f{1-\la_1^2}{1+\la_1^2}+\f{1-\la_2^2}{1+\la_2^2}=\f{2(1-\la_1^2\la_2^2)}{(1+\la_1^2)(1+\la_2^2)}\qquad \mathrm{on}\ \overline{\Om}\times[0,T].
\endaligned
\end{equation}
Let $\k,\nu$ be the constants in Theorem \ref{DfMCF}, and $\Psi\triangleq\f{2\k^2}{\nu^2}$.
There is a constant $\ep_{\varphi}>0$ depending on $n,m$, $\la_\Om$, $\sup_{\Om}|D\varphi|$ and $\sup_{\Om}|D^2\varphi|$ such that
if \eqref{000} holds, then $\sup_\Om\Th_f(\cdot,0)>\Psi^{-1}$ from Lemma \ref{Du1a},
and $v_f=\det\left(\de_{ij}+\sum_\a f^\a_if^\a_j\right)<\Psi$ on $\Om$ from Lemma \ref{smallT} in Appendix II.
Let $t_1$ be the maximal time $\le T$ such that $\Th_f\ge\Psi^{-1}$ and $v_f\le\Psi$ on $\overline{\Om}\times[0,t_1]$.
For a suitable constant $\ep_{\varphi}>0$, we can assume that
\begin{equation}\aligned\label{Df2nbdy}
\sup_{0\le t\le t_1}\left(\sup_{\p\Om}\sum_{\a=1}^{m-1}|Df^\a|^2\right)\le\f{\nu^2}{25(m-1)\k^2}
\endaligned
\end{equation}
from Lemma \ref{PDfa}, and
\begin{equation}\aligned\label{Df2nkint}
|Df^\a|\le\f{1}{m-1}\qquad on \ \ \Om_{1/\k}\times[0,t_1]
\endaligned
\end{equation}
for $\a=1,\cdots,m-1$ from Lemma \ref{KDfbound}.

Now we assume that there is a time $t_0\in(0,t_1]$ such that
\begin{itemize}
  \item $\sup_{x\in\Om}\Th_f(x,t_0)=\Psi^{-1}$\  or \ $\sup_{x\in\Om} v^2_f(x,t_0)=\Psi$;
  \item $\Th_f(x,t)>\Psi^{-1}$ and $v^2_f(x,t)<\Psi$ for all $(x,t)\in\overline{\Om}\times[0,t_0)$.
\end{itemize}

Let $E_1,\cdots,E_{n+m}$ be a standard basis of $\R^n\times\R^m$. We can see $v_f(x,t)$ as a function on graph$_{f(\cdot,t)}$ for each $t\in[0,T]$ defined by
$$v_f^{-1}=\left\lan n^1_f\bigwedge\cdots\bigwedge n^m_f, E_{n+1}\bigwedge\cdots\bigwedge E_{n+m}\right\ran,$$
where $n^1_f,\cdots,n^m_f$ are the local orthonormal normal vector fields of graph$_{f(\cdot,t)}$ in $\R^{n+m}$.
Let $\De_f$ and $A_f$ be the Laplacian and the second fundamental form of graph$_{f(\cdot,t)}$ for each $t\in[0,T]$, respectively.
Note that graph$_{f(\cdot,t)}=\{(x,f(x,t))\in\R^n\times\R^m|\ x\in\Om\}$ moves by mean curvature. Then the Gauss map of graph$_{f(\cdot,t)}$ satisfies the harmonic heat flow equation, which is a parabolic version of the Ruh-Vilms theorem. 
(2.8) in \cite{JXY} then yields
\begin{equation}\aligned\label{evologvf}
\left(\f{\p}{\p t}-\De_f\right)v_f=-v_f\left(|A_f|^2+\sum_{i,j}\la_i\la_j\left(h^i_{ik}h^j_{jk}+h^j_{ik}h^i_{jk}\right)\right),
\endaligned
\end{equation}
where $h^\a_{ij}$ are the components of the second fundamental form defined by $h^\a_{ij}=\lan\overline{\na}_{e^f_i}e^f_j,n^\a_f\ran$, $\{e^f_i\}$ is a tangent basis at the considered point in graph$_{f(\cdot,t)}$, and $\overline{\na}$ is the Levi-Civita connection of $\R^{n+m}$.
From $\Th_f\ge\Psi^{-1}$ for each $t\in[0,t_0]$, one have $\la_1\la_2\le1$, and $\Psi^{-1}\le \f{2(1-\la_1^2\la_2^2)}{4\la_1^2\la_2^2}$, which implies
$\la_1\la_2\le\sqrt{\f{\Psi}{\Psi+2}}<\f{\Psi}{\Psi+2}$.
Combining \eqref{evologvf} and $\la_i\la_j<\f{\Psi}{\Psi+2}$ for all $i\neq j$, we obtain
\begin{equation}\aligned\label{Gfparabolic}
\left(\f{\p}{\p t}-\De_f\right)v_f\le-\f{2v_f}{\Psi+2}|A_f|^2\qquad  \mathrm{for\ each}\ t\in[0,t_0].
\endaligned
\end{equation}
From Theorem \ref{DfMCF} and \eqref{Df2nkint} one has
\begin{equation}\aligned\label{pOmt0Dfm}
\sup_{\p\Om\times[0,t_0]}|Df^m|\le\f{\k}{\nu}+|D\varphi|_\Om.
\endaligned
\end{equation}
Since
\begin{equation}\aligned\label{knuDvp}
\f{\k}{\nu}\ge2\left(1+|D\varphi|_\Om\right)e^{16n|D\varphi|_\Om}\ge2\left(1+|D\varphi|_\Om\right)\left(1+16n|D\varphi|_\Om\right)\ge66|D\varphi|_\Om+2,
\endaligned
\end{equation}
then $\f{\k}{\nu}+|D\varphi|_\Om\le\f{67\k}{66\nu}$.
Combining Lemma \ref{Du1a} and \eqref{Df2nbdy}, we obtain $\la_1\la_2\le\f12$ on $\p\Om\times[0,t_0]$ and then
$$\sup_{\p\Om\times[0,t_0]}\Th_f\ge\sup_{\p\Om\times[0,t_0]}2\f{1-\la_1^2\la_2^2}{v_f^2}>\Psi^{-1}.$$
From Lemma 5.3 in \cite{TW}, $\Th_f$ satisfies the maximum principle as $\la_i\la_j<1$ on $\overline{\Om}\times[0,t_0]$ for all $i\neq j$, which implies $\sup_\Om\Th_f(\cdot,t_0)>\Psi^{-1}$.
Combining \eqref{pOmt0Dfm}\eqref{knuDvp} and Lemma \ref{smallT} one has
\begin{equation}\aligned
v^2_f\le&1+\left(\sum_{\a=1}^{m}|Df^\a|\right)^2
\le1+\left(|Df^m|+\sqrt{(m-1)\sum_{\a=1}^{m-1}|Df^\a|^2}\right)^2\\
\le&1+\left(\f{\k}{\nu}+|D\varphi|_\Om+\f{\nu}{5\k}\right)^2\le1+\left(\f{\k}{\nu}+\f{\k}{20\nu}\right)^2<\Psi
\endaligned
\end{equation}
on $\p\Om\times[0,t_0]$. With the maximum principle for \eqref{Gfparabolic}, we have $\sup_{\Om\times[0,t_0]}v^2_f<\Psi$.
Therefore, such $t_0$ does not exist provided \eqref{000} holds for such constant $\ep_{\varphi}>0$.

Note that $|f^\a|_{1+\g;\overline{\Om}\times[0,T]}\le C_{\Om,\psi}$ for each $\a=1,\cdots,m$.
By Theorem 8.2 in \cite{Li} (or Lemma 5.1 in \cite{DJX}), for each $t\in(0,T]$ the flow \eqref{MCFmainMCF} from the time $t$ with boundary data $f(\cdot,t)$ has a short-time existence on $[t,t+t']$ for some $t'>0$ depending only on $m,n,\Psi,|D\psi|_\Om,|D^2\psi|_{\Om}$ and the curvature of $\p\Om$.
Hence, $|f^\a|_{1+\g;\overline{\Om}\times[0,T+t'/2]}<\infty$ for each $\a=1,\cdots,m$.
From Theorem 3.3 in \cite{DJX}, $|f^\a|_{1+\g';\overline{\Om}\times[0,T']}\le C'_{\Om,\psi}$ for each $\a=1,\cdots,m$, where $\g'\in(0,\g)$ and $C'_{\Om,\psi}$ are constants depending only on $m,n,\Psi,|D\psi|_\Om,|D^2\psi|_{\Om}$ and the curvature of $\p\Om$.
The constant $\Psi$ depends only on $m,n,|D\psi|_\Om,|D^2\psi|_{\Om}$ and $\la_\Om$, but is independent of the time.
Hence, the flow has long-time existence and does not have any finite time singularity.
Let $H_{M_t}$ denote the mean curvature of $M_t\triangleq \mathrm{graph}_{f(\cdot,t)}$.
From $\f{\p}{\p t}v_f=-|H_{M_t}|^2v_f$ and $\p M_t=\p\Om\times\{0\}$, one has
\begin{equation}\aligned
\int_0^\infty\left(\int_{M_t}|H_{M_t}|^2\right)dt<\infty.
\endaligned
\end{equation}
From Lemma 3.1 in \cite{DJX}, we obtain the interior curvature estimates for $M_t$ with any $t>0$.
Then one can get the estimates of the higher order derivatives of $f$, which consequently are  uniformly bounded in each compact set of $\Om$.
Combining this with the above uniform boundary estimates of $M_t$, there is a sequence $t_i\rightarrow\infty$ such that $M_{t_i}$ converges to a smooth minimal graph over $\Om$ with the graphic function $u=(u^1,\cdots,u^m)\in C^{1,\g'}(\overline{\Om})$ and  $u=\psi$ on $\p\Om$. It is then standard to obtain  $W^{2,p}$-estimates (see \cite{W} for instance). From these estimates and the Sobolev imbedding theorem, $u\in C^{1,\g_*}(\overline{\Om})$ for any $\g_*\in(0,1)$. This completes the proof.
\end{proof}

\section{The Dirichlet problem on convex domains}
\label{convexdom}
We now turn our attention to the smaller class of domains that are convex, and
not only mean convex. This will allow us to obtain better bounds. In fact,
as explained in the Introduction,
the example of \cite{LO} and the Bernstein result of \cite{JXY} suggest an explicit quantitative bound for the estimates. On domains that are convex, we can indeed obtain a quantitative result in this direction stated in Theorem \ref{mainMCF*}. 
\begin{proof}
According to Theorem 8.2 in \cite{Li} (see also Lemma 5.1 in \cite{DJX} for a version in the present context), there are a constant $T>0$, $\g\in(0,1)$ and
a solution $(f^1,\cdots,f^m)\in C^\infty(\Om\times(0,T])\cap C^{1,\g}(\overline{\Om}\times[0,T])$ to the mean curvature flow
\begin{equation}
\left\{\begin{split}
\f{\p f^\a}{dt}=&g^{ij}f^\a_{ij}\qquad \mathrm{in}\ &\Om\times(0,T]\\
f^\a(\cdot,0)=&\psi^\a \qquad\ \ \mathrm{on}\ &\Om\times\{0\}\\
f^\a(\cdot,t)=&\psi^\a\qquad \mathrm{on}\ &\p\Om\times[0,T]\\
\end{split}\right.\qquad\qquad \mathrm{for}\ \a=1,\cdots,m,
\end{equation}
such that $|f^\a|_{1+\g;\overline{\Om}\times[0,T]}\le C_{\Om,\psi}$ for each $\a=1,\cdots,m$, where $C_{\Om,\psi}$ is a constant depending only on $m,n,|D\psi|_\Om,|D^2\psi|_{\Om}$, diam$(\Om)$ and the curvature of $\p\Om$.

Recall $v^2_f(x,t)=\det\left(\de_{ij}+\sum_\a f^\a_if^\a_j\right)$ for each $(x,t)\in \Om\times[0,T]$. From Lemma \ref{sssss} in Appendix II, $\sup_{x\in\Om}v^2_f(x,0)<b_0$.
We assume that there is a time $t_0\in(0,T]$ such that $\sup_{x\in\Om}v^2_f(x,t)<b_0$ for each $t\in[0,t_0)$ and $\sup_{x\in\Om}v^2_f(x,t_0)=b_0$.
Note that the right hand side of \eqref{evologvf} is the same as the right hand side of (3.7) in \cite{JXY}, and so, the estimates (3.13), (3.16), Lemma 3.1 and Lemma 3.2 in \cite{JXY} apply (those estimates are derived algebraically without using the minimal surface system). Therefore, for $t\in[0,T)$ one has
\begin{equation}\aligned\label{Gfparabolic*}
\left(\f{\p}{\p t}-\De_f\right)v_f\le0,
\endaligned
\end{equation}
where $\De_f$ is the Laplacian of graph$_f$. Since $\Om$ is convex,  we can
allow $r_{_\Om}\rightarrow\infty$ in Lemma \ref{BGEMCF}. From Lemma
\ref{BGEMCF} with \eqref{cond0} and $v^2_{f}\le b_0$ on $\Om\times[0,t_0]$, we obtain
\begin{equation}\aligned
\sum_{\a=1}^{m}\sup_{\p\Om\times[0,t_0]}|Df^\a|<\sqrt{b_0-1}.
\endaligned
\end{equation}
By Lemma \ref{sssss} in Appendix II, one has $\sup_{\p\Om\times[0,t_0]}v^2_{f}<b_0$. From the parabolic maximum principle, it follows that $\sup_{\Om\times[0,t_0]}v^2_{f}<b_0$, which is a contradiction to $\sup_{x\in\Om}v^2_f(x,t_0)=b_0$. Hence $t_0$ does not exist.
The remaining argument is similar to the last part of the proof of Theorem \ref{main0}. Hence there is a sequence $t_i\rightarrow\infty$ such that $\mathrm{graph}_{f(\cdot,t_i)}$ converges to a smooth minimal graph with the graphic function $u=(u^1,\cdots,u^m)$ and boundary data $\psi$ and $u\in C^{1,\g}(\overline{\Om})$ for any $\g\in(0,1)$. It is clear that $\sup_\Om\mathrm{det}\left(\de_{ij}+u^\a_iu^\a_j\right)\le b_0$. Combining \eqref{cond0} and Lemma \ref{PDfa}, it follows that
\begin{equation}\aligned
\sum_{\a=1}^{m}\sup_{\p\Om}|Du^\a|<\sqrt{b_0-1}.
\endaligned
\end{equation}
Let $\De_u$ denote the Laplacian of graph$_u$.
Using Lemma \ref{sssss} and the maximum principle for $\De_u\mathrm{det}\left(\de_{ij}+u^\a_iu^\a_j\right)\ge0$, we obtain $\sup_\Om\mathrm{det}\left(\de_{ij}+u^\a_iu^\a_j\right)<b_0$ and complete the proof.
\end{proof}

For any $C^2$ map $\phi: \S^{n+k}\to \S^n\subset\ir{n+1}, n\ge 4,  k>0$ which
is not homotopic to zero as a map to $\S^n$, Lawson-Osserman \cite{LO}
constructed  boundary data on the unit ball in $\ir{n}$, for which the
corresponding Dirichlet problem has no solution. Their example arises
  from the  map
  \begin{eqnarray*}u: \R^4 &\to& \R^3\\
   (z_1,z_2)\in \C^2 &\mapsto& (|z_1|^2-|z_2|^2,\;  2 z_1\bar z_2)\in \R^3
  \end{eqnarray*}
which is an extension of the classical Hopf map $\S^3\to \S^2$, and they used
it as a counterexample to the Bernstein problem in higher codimension. Since
local regularity results, underlying for instance solutions to Dirichlet problems, 
and global Bernstein theorems are related via scaling
arguments, they could then also use this example to infer non-solvability of
certain Dirichlet problems. For this particular $u$, its slope satisfies
$\sup_{\R^4}\mathrm{det}\left(\de_{ij}+u^\a_iu^\a_j\right)^{1/2}=9$, and it
is a fundamental open question whether the constant 9 occurring here is sharp
for the Bernstein problem. So far, the best value for which a general
Bernstein theorem in higher codimensions could be derived \cite{JXY} is 3
instead of 9. It is currently unclear whether one can go substantially beyond 3, because the
convex geometry of the Gauss regions in Grassmannians on which all such
results depend breaks down  for values $>3$. (Actually, as we can see by a
contradiction argument, there is a constant $t>0$ (depending on the dimension)
so that the Bernstein theorem still holds for the value $3+t$. 
In fact, if not, there is a sequence $t_i$ converging to 0, and a sequence of
nontrivial $n$-minimal graphs $M_i$ with slope $v_i<3+t_i$. We consider a
tangent cone $C_i$ of $M_i$ at infinity. The slope of $C_i$ then is also less
than $3+t_i$. Choosing a subsequence, we may assume that $C_i$ converges to a
stationary varifold $C$ in the varifold sense. It is not hard to see that $C$
has multiplicity one. Moreover, $C$ is a cone, and its slope is $\le 3$.  By
\cite{JXY}, $C$ is flat. Since, however, the $C_i$ are not flat with a
singular point 0 at least, we get a contradiction by Allard's regularity
theorem.) -- Conversely, one may also investigate
the rigidity of the Lawson-Osserman cone, see \cite{JXY2}. Since the Bernstein
problem is structurally more restricted than the Dirichlet or the local
regularity problem, it makes sense to look for the optimal constant there. Because of the relation between local
regularity and global Bernstein alluded to above, we should expect that the
optimal value of the constant $\sqrt{b_0}$ in our Theorem \ref{mainMCF*} and
the optimal value for the Bernstein problem coincide. Thus, our theorem
reaches the best known value and is
quantitatively explicit.

\section{The Dirichlet problem by perturbation}

Let $\Om$ be a bounded domain in $\R^n$ with $\p\Om\in C^2$.
Let $M$ be a smooth minimal graph over $\Om$ in $\R^{n+m}$ with the graphic function $u=(u^1,\cdots,u^m)\in C^\infty(\Om)$.
For any function $\phi=(\phi^1,\cdots,\phi^m)\in C^2(\Om)$, and a constant $\de\in(0,1]$, let $M_s$ denote a graph over $\Om$ with the graphic function $(u^1_s,\cdots,u^m_s)=(u^1+s\phi^1,\cdots,u^m+s\phi^m)$ for each $|s|\le\de$.
Let $g_{ij}^sdx_idx_j$ be the metric of $M_s$ defined by
\begin{equation}\aligned\label{gijs}
g_{ij}^s=\de_{ij}+\sum_\a \p_{x_i}u^\a_s\p_{x_j}u^\a_s.
\endaligned
\end{equation}
Let $(g^{ij}_s)$ be the inverse matrix of $(g^s_{ij})$.
Let $v_{u_s}=\sqrt{\mathrm{det}g_{ij}^s}$, and $\De_{M_s}$ be the Laplacian of $M_s=\mathrm{graph}_{u_s}$.
We omit the index $s$ for $s=0$, and write $\p_i$ instead of $\p_{x_i}$ for convenience.
Let $L$ be the linear differential operator of the second order defined by $L\phi=(L\phi^\a,\cdots,L\phi^m)$ with
\begin{equation}\aligned\label{LphiDes}
L\phi^\a=\f{\p}{\p s}\bigg|_{s=0}\De_{M_s}u_s^\a.
\endaligned
\end{equation}
Note that $\De_Mu^\a=0$.
A direct computation implies
\begin{equation}\aligned\label{LphiDes*}
L\phi^\a=&\f1{v_{u}}\f{\p}{\p s}\bigg|_{s=0}\p_{i}\left(g^{ij}_sv_{u_s}\p_{j}u^\a_s\right)\\
=&\f1{v_{u}}\p_i\left(g^{ij}\phi_j^\a v_u+u_j^\a\left(-2g^{ik}g^{lj}u^\be_{k}\phi^\be_{l}+g^{ij}g^{kl}u^\be_{k}\phi^\be_{l}\right)v_u\right).
\endaligned
\end{equation}
\begin{remark}
If we further assume $u^\a,\phi^\a\in C^1(\overline{\Om})$, $\phi^\a=0$ on $\p\Om$ for all $\a=1,\cdots,m$,
then $-\int_\Om\lan \phi,L\phi\ran v_u$ has the following geometric meaning.
\begin{equation}\aligned
-\int_\Om\lan \phi,L\phi\ran v_u=&\f{\p}{\p s}\bigg|_{s=0}\int_\Om\left\lan \phi^\a,\p_{i}\left(g^{ij}_sv_{u_s}\p_{j}u^\a_s\right)\right\ran\\
=&\f{\p}{\p s}\bigg|_{s=0}\int_\Om g^{ij}_s\phi^\a_i\p_{j}u^\a_sv_{u_s}
=\f{\p^2}{\p s^2}\bigg|_{s=0}\int_\Om v_{u_s}.
\endaligned
\end{equation}
In particular, for the codimension $m=1$,
\begin{equation}\aligned
&-\int_M\lan \phi,L\phi\ran=\int_M\left(|\na_M\phi|^2-|\lan\na_M w,\na_M\phi\ran|^2\right),
\endaligned
\end{equation}
where $\na_M$ is the Levi-Civita connection of $M=\mathrm{graph}_u$.
\end{remark}
It's not hard to see that there is a constant $c_n>0$ depending only on $n$ such that
\begin{equation}\aligned
\left|\f{\p^2}{\p s^2}\De_{M_s}u_s^\a\right|\le c_n(1+|Du|^2+|D\phi|^2)^3\left(|D^2u|+|D^2\phi|\right).
\endaligned
\end{equation}
Moreover, suppose $\hat{\phi}=(\hat{\phi}^1,\cdots,\hat{\phi}^m)$, $\hat{u}_s=(\hat{u}^1_s,\cdots,\hat{u}^m_s)=(u^1+s\hat{\phi}^1,\cdots,u^m+s\hat{\phi}^{m})$ for each $0\le s\le\de$, and $\hat{u}_0=(u^1,\cdots,u^m)$. Denote $\hat{M}_s=\mathrm{graph}_{\hat{u}_s}$. Then with a suitable constant $c_n>0$
\begin{equation}\aligned\label{w-hatw*}
&\left|\f{\p^2}{\p s^2}\left(\De_{M_s}u_s^\a-\De_{\hat{M}_s}\hat{u}_s^\a\right)\right|\le c_n(1+|Du|^2+|D\phi|^2+|D\hat{\phi}|^2)^3|D^2(\phi-\hat{\phi})|\\
&+c_n|D(\phi-\hat{\phi})|(1+|Du|^2+|D\phi|^2+|D\hat{\phi}|^2)^{5/2}\left(|D^2u|+|D^2\phi|+|D^2\hat{\phi}|\right).
\endaligned
\end{equation}
Let $Q_{\de,u,\phi}=(Q^1_{\de,u,\phi},\cdots,Q^m_{\de,u,\phi})$ with
\begin{equation}\aligned\label{Qadewphi}
Q^\a_{\de,u,\phi}=-\f1{\de}\int_0^\de\left(\int_0^\tau\f{\p^2}{\p s^2}\De_{M_s}u_s^\a ds\right) d\tau.
\endaligned
\end{equation}
Then for each $\a=1,\cdots,m$
\begin{equation}\aligned\label{BdQadeuphi}
\left|Q^\a_{\de,u,\phi}\right|\le c_n\de(1+|Du|^2+|D\phi|^2)^3\left(|D^2u|+|D^2\phi|\right).
\endaligned
\end{equation}
From \eqref{w-hatw*}, we have
\begin{equation}\aligned\label{w-hatw}
&\left|Q^\a_{\de,u,\phi}-Q^\a_{\de,u,\hat{\phi}}\right|\le c_n\de(1+|Du|^2+|D\phi|^2+|D\hat{\phi}|^2)^3|D^2(\phi-\hat{\phi})|\\
&+c_n\de|D(\phi-\hat{\phi})|(1+|Du|^2+|D\phi|^2+|D\hat{\phi}|^2)^{5/2}\left(|D^2u|+|D^2\phi|+|D^2\hat{\phi}|\right).
\endaligned
\end{equation}
Suppose that $M_\de$ is a minimal graph, then
\begin{equation}\aligned\label{DeMdew0}
0=\De_{M_\de}u_\de^\a-\De_{M_0} u_0^\a=\int_0^\de\f{\p}{\p s}\De_{M_s}u_s^\a=&\int_0^\de\left(\int_0^\tau\f{\p^2}{\p s^2}\De_{M_s}u_s^\a ds+L\phi^\a\right) d\tau\\
=&-\de Q^\a_{\de,u,\phi}+\de L\phi^\a,
\endaligned
\end{equation}
which implies
\begin{equation}\aligned\label{DeMdew0*}
L\phi^\a=Q^\a_{\de,u,\phi}.
\endaligned
\end{equation}

Now we assume that $\Om$ is mean convex. Let $\k_\Om$ denote the largest absolute principle curvature of $\p\Om$, i.e.,
$\k_\Om=\sup_{\p\Om}\sup_{i=1,\cdots,n-1}|\la_i(D^2d)|$, where $d=d(\cdot,\p\Om)$ on $\overline{\Om}$.
For any $\varphi\in C^2(\overline{\Om})$,
from Jenkins-Serrin \cite{JS} (see also Theorem 16.8 and Theorem 14.9 in \cite{GT}), there is a smooth solution $w$ to the minimal surface equation \eqref{MS0} with $w=\varphi$ on $\p\Om$ such that
$\sup_\Om|Dw|\le C_{\Om,\varphi}$ for a constant $C_{\Om,\varphi}>0$ depending only on $n,\mathrm{diam}\Om,\sup_\Om(|D\varphi|+|D^2\varphi|)$.
From Theorem 13.7 in \cite{GT} and $W^{2,p}$-estimates (see \cite{W} for instance), for any $p>n$ there is a constant $C_{p,\Om,\varphi}$ depending only on $n,p,\mathrm{diam}\Om,\k_\Om,\sup_\Om(|D\varphi|+|D^2\varphi|)$ such that
\begin{equation}\aligned
||w||_{W^{2,p}(\Om)}\triangleq \int_\Om\left(|w|^p+|Dw|^p+|D^2w|^p\right)\le C_{p,\Om,\varphi}.
\endaligned
\end{equation}
By the Sobolev embedding theorem, we can assume the H\"older norm $|Dw|_{\g,\Om}\le C_{p,\Om,\varphi}$ with $\g=1-n/p$.

Suppose $(u^1,\cdots,u^m)=(0,\cdots,0,w)$. In this case $g_{ij}=g^0_{ij}=\de_{ij}+w_iw_j$ from \eqref{gijs}. Moreover, $g^{ij}=\de_{ij}-w_iw_jv_w^{-2}$, and $g^{ij}w_j=w_iv_w^{-2}$.
From \eqref{LphiDes*}, one has
\begin{equation}\aligned\label{Lphiuuuu}
L\phi^\a=\f1{v_{w}}\p_i\left(g^{ij}\phi_j^\a v_w+w_j^\a\left(-2g^{lj}w_i\phi^m_{l}+g^{ij}w_{l}\phi^m_{l}\right)v_w^{-1}\right).
\endaligned
\end{equation}
Let $L_*$ be a linear differential operator of the second order defined by
\begin{equation}\aligned\label{L*DEF}
L_*\xi=\f1{v_{w}}\p_i\left(v_w\left(\de_{ij}-w_iw_j(v_w^{-2}+v_w^{-4})\right)\xi\right)
\endaligned
\end{equation}
for any $\xi\in C^2(\Om)$. Then from \eqref{Lphiuuuu} we have
\begin{equation}\label{LphiElliptic}
L\phi=(L\phi^1,\cdots,L\phi^m)=(\De_M\phi^1,\cdots,\De_M\phi^{m-1},L_*\phi^m).
\end{equation}

Now we shall use the Schauder fixed point theorem to show the existence of minimal graphs by perturbation of a given minimal graph of codimension one.
\begin{theorem}\label{ep-Ex}
For any mean convex bounded $C^2$ domain $\Om\subset\R^n$ with diameter $l=1$, $m\ge2$, $p>n$ and any $\varphi\in C^2(\overline{\Om})$, there is a constant $\de_{m,p,\Om,\varphi}>0$ depending on $n,m,p$, $\k_\Om$, $\sup_{\Om}|D\varphi|$ and $\sup_{\Om}|D^2\varphi|$ such that if the functions $\psi^1,\cdots,\psi^{m-1}\in C^{1,\g}(\overline{\Om})$ with $\g=1-n/p$ satisfy
$\sum_{\a=1}^{m-1}||\psi^\a||_{W^{2,p}(\Om)}\le \de_{m,p,\Om,\varphi}$,
then there is a solution $u=(u^1,\cdots,u^m)\in C^\infty(\Om)\cap C^{1,\g}(\overline{\Om})$ to the minimal surface system \eqref{000a} with $u=(\psi^1,\cdots,\psi^{m-1},\varphi)$ on $\p\Om$.
\end{theorem}
\begin{proof}
From the above argument, there is a smooth solution $w$ to the minimal surface equation \eqref{MS0} with boundary $\varphi$ such that
$\sup_\Om|Dw|+|Dw|_{\g,\Om}+||w||_{W^{2,p}(\Om)}\le C_{p,\Om,\varphi}$,
where $C_{p,\Om,\varphi}$ is a constant depending only on $n,p,\mathrm{diam}\Om,\k_\Om,\sup_\Om(|D\varphi|+|D^2\varphi|)$, $\g=1-n/p$.
Let $M$ be a graph over $\Om$ in $\R^{n+m}$ with the graphic function $(0,\cdots,0,w)$.
Let $\de$ be a constant in $(0,1)$ to be defined later, and $\psi^1,\cdots,\psi^{m-1}\in C^{1,\g}(\overline{\Om})$ with $\g=1-n/p$ such that
$\sum_{\a=1}^{m-1}||\psi^\a||_{W^{2,p}(\Om)}=\de^2$. Put $\hat{\psi}=(\hat{\psi}^1,\cdots,\hat{\psi}^m)=(\psi^1/\de,\cdots,\psi^{m-1}/\de,0)$.
For any $t>0$, let
$$\mathfrak{S}_t=\left\{\phi=(\phi^1,\cdots,\phi^m)\in C^{1,\g}(\overline{\Om},\R^m)\Big| \sum_\a||\phi^\a||_{W^{2,p}(\Om)}\le t\right\},$$
which is a compact set in the Banach space $W^{2,p}(\Om,\R^m)$ for any
  $t>0$. Moreover, $\mathfrak{S}_t$ is convex by the Minkowski inequality.

For any $\phi=(\phi^1,\cdots,\phi^m)\in \mathfrak{S}_1$, let $w_s=(w^1_s,\cdots,w^m_s)=(s\phi^1,\cdots,s\phi^{m-1},w+s\phi^m)$, and $M_s=\mathrm{graph}_{w_s}$ for $s\in[0,\de]$.
Let $Q^\a_{\de,w,\phi}$ be defined as \eqref{Qadewphi}, where $u$ is replaced by $(0,\cdots,0,w)$, $u^\a_s$ is replaced by $w^\a_s$. For any $\a=1,\cdots,m-1$, there is a unique solution $\xi^\a\in C^\infty(\Om)\cap C^{1,\g}(\overline{\Om})$ to
\begin{equation}
\left\{\begin{split}
\De_M\xi^\a =& Q^\a_{\de,w,\phi}\qquad\ \ &\mathrm{in}\ \Om\\
\xi^\a =&\hat{\psi}^\a\qquad &\mathrm{on}\ \p \Om\\
\end{split}\right.,
\end{equation}
such that
\begin{equation}\aligned\label{xiaQhpsi}
||\xi^\a||_{W^{2,p}(\Om)}\le C_{p,\Om,\varphi}\left(||Q^\a_{\de,w,\phi}||_{L^p(\Om)}+||\hat{\psi}^\a||_{W^{2,p}(\Om)}\right).
\endaligned
\end{equation}
where $C_{p,\Om,\varphi}$ is a general constant depending only on $n,p,\mathrm{diam}\Om,\k_\Om,\sup_\Om(|D\varphi|+|D^2\varphi|)$. From \eqref{Qadewphi} and the definition of $\hat{\psi}$, for a suitable constant $C_{p,\Om,\varphi}$ we get
\begin{equation}\aligned\label{xiaW2p}
||\xi^\a||_{W^{2,p}(\Om)}\le C_{p,\Om,\varphi}\de.
\endaligned
\end{equation}
Let $L_*$ be the operator defined in \eqref{L*DEF}, then
$L_*$ is uniformly elliptic since $1-|Dw|^2(v_w^{-2}+v_w^{-4})=v_w^{-4}$.
Then for a suitable constant $C_{p,\Om,\varphi}$, there is a unique solution $\xi^m\in C^\infty(\Om)\cap C^{1,\g}(\overline{\Om})$ to
\begin{equation}
\left\{\begin{split}
L_*\xi^m =& Q^m_{\de,w,\phi}\qquad\ \ &\mathrm{in}\ \Om\\
\xi^m =&0\qquad &\mathrm{on}\ \p \Om\\
\end{split}\right.,
\end{equation}
such that
\begin{equation}\aligned\label{ximW2p}
||\xi^m||_{W^{2,p}(\Om)}\le C_{p,\Om,\varphi}||Q^m_{\de,w,\phi}||_{L^p(\Om)}\le C_{p,\Om,\varphi}\de.
\endaligned
\end{equation}

Let $T_{\psi}$ be the operator defined by letting $\xi=(\xi^1,\cdots,\xi^m)=T_{\psi}\phi$ be the unique solution in $W^{2,p}(\Om)$ of the following linear Dirichlet problem,
\begin{equation}\label{Dzxi}
\left\{\begin{split}
L\xi^\a=& Q^\a_{\de,w,\phi} \qquad\ \ &\mathrm{in}\ \Om\\
\xi^\a=&\hat{\psi}^\a\qquad\qquad &\mathrm{on}\ \p \Om\\
\end{split}\right.
\end{equation}
for each $\a=1,\cdots,m$.
Combining \eqref{xiaW2p} and \eqref{ximW2p}, we have
\begin{equation}\aligned\label{xiW2p}
\sum_{\a=1}^m||\xi^\a||_{W^{2,p}(\Om)}\le mC_{p,\Om,\varphi}\de.
\endaligned
\end{equation}
Put $\de=\f1{mC_{p,\Om,\varphi}}$, then $\xi\in \mathfrak{S}_1$.
Combining this with \eqref{w-hatw}, $T_{\psi}$ is a continuous mapping from $\mathfrak{S}_1$ into $\mathfrak{S}_1$.
From \eqref{BdQadeuphi}\eqref{xiaQhpsi}\eqref{ximW2p}, $T_\psi$ is a compact operator, i.e., $T_\psi(K)$ is precompact for any compact $K$ in $W^{2,p}(\Om,\R^m)$.
Hence by the Schauder fixed point theorem (see also \cite{GT}), there is a fixed point $\phi_*=(\phi_*^1,\cdots,\phi_*^m)\in\mathfrak{S}_1$ for the operator $T_{\psi}$. Namely,
$$\phi_*=T_{\psi}\phi_*.$$
Let $u^\a=\de\phi_*^\a$ for $\a=1,\cdots,m-1$ and $u^m=\de\phi_*^m+w$. From \eqref{DeMdew0}\eqref{DeMdew0*}, $(u^1,\cdots,u^m)$ is a smooth solution to the minimal surface system \eqref{000a} with $u^m=\de\phi_*^\a+w=\varphi$ and $u^\a=\de\phi_*^\a=\de\hat{\psi}=\psi^\a$ on $\p\Om$ for $\a=1,\cdots,m-1$.
\end{proof}

\section{Non-existence results for solutions of Dirichlet problems}

\begin{theorem}\label{NONEXIST}
Let $\Om$ be a bounded domain in $\R^n(n\ge3)$ with a smooth mean convex boundary, and suppose that there is a point $q\in\p\Om$ such that $\p\Om$ is not convex, but has zero mean curvature at $q$. Then for any constant $0<\ep\le1$, there exists a vector-valued function $\psi=(\psi^1,\psi^2)\in C^2(\overline{\Om},\R^2)$ with $|D\psi^1|\le\ep$ such that the minimal surface system \eqref{000a} has no classical solution with boundary data $(\psi^1,\psi^2)$.
\end{theorem}
\begin{proof}
Without loss of generality, we may assume that $q$ is the origin, the unit normal vector $\nu$ at 0 to $\p\Om$ is parallel to the axis $x_n$, and $\lan D_{e_1}e_1,\nu\ran<0$ at $0$. Here, $e_1$ is a unit tangent vector field of $\p\Om$ in a neighborhood of the origin, such that $e_1$ is parallel to the axis $x_1$ at 0.

Let $\psi$ be a linear function on $\overline{\Om}$ such that $D\psi=\ep E_1$ for some fixed constant $0<\ep\le1$. Here, $\{E_i\}_{i=1}^{n-1}$ is a standard basis of $\R^n$ such that $E_n$ is parallel to the axis $x_n$.
Let $\varphi$ be a smooth function on $\overline{\Om}$ to be defined later.
Assume that there is a smooth solution $(u,v)$ of the minimal surface system
\begin{equation}\label{nDMG}
\left\{\begin{split}
g^{ij}u_{ij}=g^{ij}v_{ij}=0\qquad \mathrm{in}\ \Om\\
u=\varphi,\ \ v=\psi\qquad \mathrm{on}\ \p\Om\\
\end{split}\right.\qquad\qquad\qquad
\end{equation}
with $g_{ij}=\de_{ij}+ u_iu_j+v_iv_j$.

Let $p_{ij}=\de_{ij}+u_i u_j+\psi_i\psi_j$, and $(p^{ij})$ be the inverse matrix of $(p_{ij})$. Note that $D^2\psi\equiv0$, which means that
$$\sum_{i,j=1}^np^{ij}\psi_{ij}=0.$$
Then by the maximum principle, $v=\psi$ on $\overline{\Om}$.

Put $P_{ij}=\de_{ij}+\psi_i \psi_j+w_iw_j$ for some function $w\in C^2(\overline{\Om})$. Then
\begin{equation}\aligned\nonumber
P_{11}=1+\ep^2+w_1^2,
\endaligned
\end{equation}
and
\begin{equation}\aligned\nonumber
P_{ij}=\de_{ij}+w_iw_j,\qquad \mathrm{for}\ 1\le i,j\le n,\ i+j\neq2.
\endaligned
\end{equation}
Set
$$|w|_{\ep,1}^2=\f{w_1^2}{1+\ep^2}+\sum_{i=2}^{n}w_i^2.$$
The components of the inverse matrix of $(P_{ij})$ are
\begin{equation}\aligned\label{P11i}
P^{11}=\f1{1+\ep^2}-\f{w_1^2}{(1+|w|^2_{\ep,1})(1+\ep^2)^2},\qquad P^{i1}=-\f{w_1w_i}{(1+|w|^2_{\ep,1})(1+\ep^2)}\qquad \mathrm{for}\ i\ge2,
\endaligned
\end{equation}
and
\begin{equation}\aligned\label{Pij}
P^{ij}=\de_{ij}-\f{w_iw_j}{1+|w|^2_{\ep,1}},\qquad \mathrm{for}\ 2\le i,j\le n.
\endaligned
\end{equation}
One can check this easily as follows. We see $P_{ij}$ and $P^{ij}$ as smooth functions of $\ep^2$. Let $P_{11}(t)=1+t+w_1^2$ and $P_{ij}(t)=\de_{ij}+w_iw_j$ for all $t\ge0$, then $(P_{11})'=\f{\p}{\p t}P_{11}=1$ and $(P_{ij})'=0$ for all $i,j\ge2$. Let $(P^{ij})=(P^{ij}(t))$ be the inverse matrix of $(P_{ij}(t))$, then
\begin{equation}\aligned\label{Piji1j1}
(P^{ij})'=-P^{i1}P^{j1}.
\endaligned
\end{equation}
We solve ODE \eqref{Piji1j1} and get $P^{11}(t)=P^{11}(0)/(1+P^{11}(0)t)$, $P^{i1}(t)=P^{i1}(0)/(1+P^{11}(0)t)$ for all $i\ge2$, $P^{ij}(t)=P^{ij}(0)-P^{i1}(0)P^{j1}(0)t/(1+P^{11}(0)t)$ for all $i,j\ge2$. Taking $t=\ep^2$, we can show \eqref{P11i} and \eqref{Pij}.

Let $d(x)=d(x,\p\Om)$ for  $x\in\overline{\Om}$ as before. Set $d_{ij}=\p_{x_i}\p_{x_j}d$ for $1\le i,j\le n$. Let $\{e_i\}_{i=1}^{n-1}$ be a local orthonormal basis in a neighborhood of the origin, such that $e_1$ is parallel to the axis $x_1$ at 0. Set $\Om_t=\{x\in\Om|\ d(x,\Om)>t\}$ for $t\ge0$ as before. Let $e_n$ be the unit normal vector field to $\p\Om_t$ so that $e_n$ points into $\Om_t$.
Since $d$ is a constant on $\p\Om_t$, then at the point $0\in\p\Om$ we get
$\left(D_{e_1}D_{e_1}-(D_{e_1}e_1)^T\right)d=0$, and
$$\p_{x_1}\p_{x_1}d=d_{11}=(D_{e_1}e_1)^Td=(D_{e_1}e_1)d-\lan D_{e_1}e_1,e_n\ran D_{e_n}d=\lan D_{e_1}e_1,e_n\ran<0.$$
Since $\De d=H_{\p\Om}$ on $\p\Om$, from the assumption there is a positive constant $a>0$ such that
\begin{equation}\aligned\label{d11Ded}
d_{11}<-a\qquad \mathrm{and} \qquad \De d\ge-\f{a\ep^2}{2(1+\ep^2)}
\endaligned
\end{equation}
on $B_{a}(0)\cap\Om$.

Let $\chi$ be a $C^2$-function on $(0,a)$ such that $\chi(2a)=0$, $\chi'\le0$, $\chi'(0)=-\infty$. Let $\chi_\de(t)=\chi\left(t-\de\right)$ for any $t\in(\de,a)$.
Set $\Om_{\de,a}=\{x\in B_a(0)\cap\Om|\ \de<d(x)<a\}$, and
$$w(x)=\chi_\de(d(x))+\sup_{|y|=a}u(y)\qquad \mathrm{for}\ x\in \Om_{\de,a}.$$
Put $|d|^2_{\ep,1}=\f{d_1^2}{1+\ep^2}+\sum_{i=2}^{n}d_i^2$, then $|w|^2_{\ep,1}=(\chi_\de')^2|d|^2_{\ep,1}$. For such $w$, from \eqref{P11i}\eqref{Pij} we have
\begin{equation}\aligned\label{P11w}
P^{11}=\f1{1+\ep^2}-\f{(\chi_\de')^2d^2_1}{\left(1+(\chi_\de')^2|d|^2_{\ep,1}\right)(1+\ep^2)^2},
\endaligned
\end{equation}
and for $2\le i,j\le n$
\begin{equation}\aligned\label{P1iijw}
P^{i1}=-\f{(\chi_\de')^2d_1d_i}{\left(1+(\chi_\de')^2|d|^2_{\ep,1}\right)(1+\ep^2)}, \qquad P^{ij}=\de_{ij}-\f{(\chi_\de')^2d_id_j}{1+(\chi_\de')^2|d|^2_{\ep,1}}.
\endaligned
\end{equation}
So one has
\begin{equation}\aligned\label{6.4}
\sum_{i,j=1}^nP^{ij}d_id_j=&\f{d_1^2}{1+\ep^2}-\f{(\chi_\de')^2d^4_1}{\left(1+(\chi_\de')^2|d|^2_{\ep,1}\right)(1+\ep^2)^2}
-2\sum_{i=2}^n\f{(\chi_\de')^2d_1^2d_i^2}{\left(1+(\chi_\de')^2|d|^2_{\ep,1}\right)(1+\ep^2)}\\
&+\sum_{i,j=2}^n\left(\de_{ij}d_id_j-\f{(\chi_\de')^2d_i^2d_j^2}{1+(\chi_\de')^2|d|^2_{\ep,1}}\right)\\
=&|d|_{\ep,1}^2-\f{(\chi_\de')^2|d|_{\ep,1}^4}{1+(\chi_\de')^2|d|^2_{\ep,1}}=\f{|d|_{\ep,1}^2}{1+(\chi_\de')^2|d|^2_{\ep,1}}\le\f{1}{(\chi_\de')^2}.
\endaligned
\end{equation}
Note that $\sum_{i=1}^nd_{ij}d_i=\f12D_j|Dd|^2=0$.
Then
\begin{equation}\aligned
\sum_{i,j=1}^nP^{ij}d_{ij}=&\f{d_{11}}{1+\ep^2}-\f{(\chi_\de')^2d_{11}d^2_1}{\left(1+(\chi_\de')^2|d|^2_{\ep,1}\right)(1+\ep^2)^2}
-2\sum_{i=2}^n\f{(\chi_\de')^2d_{1i}d_1d_i}{\left(1+(\chi_\de')^2|d|^2_{\ep,1}\right)(1+\ep^2)}\\
&+\sum_{i,j=2}^n\left(\de_{ij}d_{ij}-\f{(\chi_\de')^2d_{ij}d_id_j}{1+(\chi_\de')^2|d|^2_{\ep,1}}\right)\\
=\f{d_{11}}{1+\ep^2}&+\sum_{i=2}^nd_{ii}-\f{(\chi_\de')^2}{1+(\chi_\de')^2|d|^2_{\ep,1}}\left(\f{d_{11}d^2_1}{(1+\ep^2)^2}
+2\sum_{i=2}^n\f{d_{1i}d_1d_i}{1+\ep^2}-\sum_{i=2}^nd_{1i}d_1d_i\right)\\
=\f{d_{11}}{1+\ep^2}&+\sum_{i=2}^nd_{ii}-\f{(\chi_\de')^2d_{11}d^2_1}{1+(\chi_\de')^2|d|^2_{\ep,1}}\f{\ep^4}{(1+\ep^2)^2}\\
\endaligned
\end{equation}
Combining this with \eqref{d11Ded} we have
\begin{equation}\aligned\label{6.6}
\sum_{i,j=1}^nP^{ij}d_{ij}\ge\f{d_{11}}{1+\ep^2}+\sum_{i=2}^nd_{ii}=-\f{\ep^2}{1+\ep^2}d_{11}+\De d>\f{\ep^2a}{1+\ep^2}-\f{a\ep^2}{2(1+\ep^2)}=\f{a\ep^2}{2(1+\ep^2)}
\endaligned
\end{equation}
on $B_a(0)\cap\Om$.
Set
$$\chi(d)=\f{2\sqrt{1+\ep^2}}{\sqrt{a}\ep}\left(\sqrt{2a}-\sqrt{d}\right),$$
then $\chi'=-\f{\sqrt{1+\ep^2}}{\sqrt{a}\ep\sqrt{d}}$ and $\chi''=\f{\sqrt{1+\ep^2}}{2\sqrt{a}\ep d^{\f32}}\ge0$ on $(0,a)$. Hence on $\Om_{\de,a}$
\begin{equation}\aligned
\sum_{i,j=1}^nP^{ij}\p_{x_i}\p_{x_j}\chi_\de=&\chi_\de'\sum_{i,j=1}^nP^{ij}d_{ij}+\chi_\de''\sum_{i,j=1}^nP^{ij}d_{i}d_j
<\chi_\de'\f{a\ep^2}{2(1+\ep^2)}+\f{\chi_\de''}{(\chi_\de')^2}\\
=&-\f{\sqrt{1+\ep^2}}{\sqrt{a}\ep\sqrt{d-\de}}\f{a\ep^2}{2(1+\ep^2)}+\f{\sqrt{a}\ep}{2\sqrt{1+\ep^2}}\f1{\sqrt{d-\de}}=0.
\endaligned
\end{equation}
Since $\chi_\de'=-\infty$ on $\p\Om_{\de,a}\setminus\p B_a(0)$, then by Theorem 13.10 in \cite{GT}, we have
$$u\le \chi_\de+\sup_{|y|=a}u(y)\qquad \mathrm{on}\ \Om_{\de,a}.$$
Letting $\de\rightarrow0$, then
\begin{equation}\aligned\label{ulechi}
u\le \f{2\sqrt{2(1+\ep^2)}}{\ep}+\sup_{|y|=a}u(y)\qquad \mathrm{on}\ B_a(0)\cap\Om.
\endaligned
\end{equation}

Let $\r(x)=|x|$, and $l=\mathrm{diam}(\Om)$. Then for every $i\in\{1,\cdots,n\}$,
$$0\le \r_{ii}\triangleq\p_{x_i}\p_{x_i}\r=\f1{|x|}-\f{x_i^2}{|x|^3}\le\f1{|x|},$$
and $\De\r=\f{n-1}{|x|}$.
Choose $\phi\in C^2((a,l))$ such that $\phi(l)=0$, $\phi'\le0$ and $\phi'(a)=-\infty$.
Set
$$w(x)=\phi(\r(x))+\sup_{\p\Om\setminus B_a(0)}u\qquad \mathrm{for\ any}\ x\in\Om\setminus B_a(0).$$
Note that $|D\r|=1$, then on $\Om\setminus \overline{B_a(0)}$, analogously to the proof of \eqref{6.4}-\eqref{6.6}, one has
\begin{equation}\aligned
\sum_{i,j=1}^nP^{ij}\r_i\r_j\le\f{1}{1+(\phi')^2},
\endaligned
\end{equation}
and
\begin{equation}\aligned
\sum_{i,j=1}^nP^{ij}\r_{ij}=&\f{\r_{11}}{1+\ep^2}+\sum_{i=2}^n\r_{ii}-\f{(\phi')^2\r_{11}\r^2_1}{1+(\phi')^2(\r_1^2+(1+\ep^2)\sum_{2\le i\le n}\r_i^2)}\f{\ep^4}{1+\ep^2}\\
\ge&-\f{\ep^2}{1+\ep^2}\r_{11}+\De \r-\f{\ep^4}{1+\ep^2}\r_{11}\ge-\f{\ep^2}{|x|}+\f{n-1}{|x|}.
\endaligned
\end{equation}
By assumption $\ep\le1$ and $n\ge3$, we have
\begin{equation}\aligned
\sum_{i,j=1}^nP^{ij}\r_{ij}\ge\f{1}{\r}.
\endaligned
\end{equation}
Let
$$\phi(\r)=\int_\r^l\left(\log\f ta\right)^{-\f12}dt.$$
Then $\phi'=-\left(\log\f{\r}a\right)^{-\f12}$ and $\phi''=\f12\left(\log\f{\r}a\right)^{-\f32}\f1\r$.
On $\Om\setminus \overline{B_a(0)}$, we have
\begin{equation}\aligned
\sum_{i,j=1}^nP^{ij}\phi_{ij}=&\phi'\sum_{i,j=1}^nP^{ij}\r_{ij}+\phi''\sum_{i,j=1}^nP^{ij}\r_{i}\r_j
\le\f{\phi'}{\r}+\f{\phi''}{(\phi')^2}\\
=&-\left(\log\f {\r}a\right)^{-\f12}\f1{\r}+\f1{2\r\sqrt{\log\f{\r}a}}<0.
\endaligned
\end{equation}
Since $\phi'(a)=-\infty$, then by Theorem 13.10 in \cite{GT}, we have
$$u\le \phi+\sup_{\p\Om\setminus B_a(0)}u(y)\qquad \mathrm{on}\ \Om\setminus B_a(0).$$
Combining \eqref{ulechi}, we obtain
\begin{equation}\aligned
u\le \f{2\sqrt{2(1+\ep^2)}}{\ep}+\phi(a)+\sup_{\p\Om\setminus B_a(0)}u(y)\qquad \mathrm{on}\ \Om\cap B_a(0).
\endaligned
\end{equation}
Hence on $\p\Om\cap B_a(0)$, $u$ cannot be arbitrary. For instance, if there is a point $x\in\p\Om\cap B_a(0)$ such that
$$\varphi(x)>\f{2\sqrt{2(1+\ep^2)}}{\ep}+\phi(a)+\sup_{\p\Om\setminus B_a(0)}\varphi(y),$$
then the minimal surface system \eqref{nDMG} has no  classical solution.
\end{proof}

\section{Uniqueness of solutions of Dirichlet problems}

In this section, we will study the uniqueness of strictly stable minimal submanifolds with fixed boundary. For simplicity, we assume that $M$ is an $n$-dimensional smooth minimal submanifold in $\R^{n+m}$ with smooth boundary.
Let $\{e_i\}_{i=1}^n$ be a local orthonormal frame field of $M$ at the considered point. Let $A_{ij}=\overline{\na}_{e_i}e_j-\na_{e_i}e_j$ be the components of the second fundamental form of $M$.
Let $\Phi$ be a smooth vector field in the normal space $NM$. 
Assume that $\ep_0$ is sufficiently small, such that the hypersurface $M_s=M+s\Phi$, given by
\begin{equation}\aligned\nonumber
\mathrm{F}(\cdot,s):\ M\rightarrow\R^{n+m}\quad \mathrm{with}\ \ \ \mathrm{F}(p,s)=p+s\Phi(p)
\endaligned
\end{equation}
is smooth for every $|s|<\ep_0$. Let $H_s$ be the mean curvature vector of $M_s$, and $\De^\bot_M$ be the normal Laplacian on $M$ for the normal bundle $NM$.
By a standard computation,
\begin{equation}\aligned\label{varHs}
\f{\p}{\p s}H_s\Big|_{s=0}=\De^\bot_M\Phi+\sum_{i,j}\lan\Phi,A_{ij}\ran A_{ij}.
\endaligned
\end{equation}
Furthermore, from the appendix I,
\begin{equation}\aligned\label{Qs}
\f1sH_s=\De^\bot_M\Phi+\sum_{i,j}\lan\Phi,A_{ij}\ran A_{ij}+s\mathcal{Q}_s\left(\Phi,\na^\bot\Phi,(\na^\bot)^2\Phi\right),
\endaligned
\end{equation}
where $\mathcal{Q}_{s,\Phi}\triangleq\mathcal{Q}_s\left(\Phi,\na^\bot\Phi,(\na^\bot)^2\Phi\right)$ is a vector-valued function defined by the following combination
\begin{equation}\aligned
\mathcal{Q}_{s,\Phi}=&Y_0*\Phi*\Phi+Y_1*\Phi*\na^\bot\Phi+Y_2*\na^\bot\Phi*\na^\bot\Phi\\
&+Y_3*\Phi*(\na^\bot)^2\Phi+Y_4*\na^\bot\Phi*(\na^\bot)^2\Phi.
\endaligned
\end{equation}
Here, $Y_i$ are smooth vector fields depending on $s$, $\Phi$, $\na^\bot\Phi$, $(\na^\bot)^2\Phi$, and $\na^\bot$ is the normal connection. Hence, if $\sum_{i=0}^2\left|(\na^\bot)^i\Phi\right|$ is bounded and $s$ is sufficiently small, then $\sum_{i=0}^4|Y_i|$ is bounded, and
\begin{equation}\aligned\label{Qs*}
\left|\mathcal{Q}_{s,\Phi}\right|\le C_M\left(|\Phi|^2+|\na^\bot\Phi|^2+\big(|\Phi|+|\na^\bot\Phi|\big)\big|(\na^\bot)^2\Phi\big|\right),
\endaligned
\end{equation}
where $C_M$ depends only on the bounds of $R_M$, $\na R_M$ and $\na^2R_M$. Here, $R_M$ is the curvature tensor of $M$. Let $L_{M}$ be the second order operator defined by
\begin{equation}\aligned
L_M\xi=\De^\bot_M\xi+\sum_{i,j}\lan\xi,A_{ij}\ran A_{ij}
\endaligned
\end{equation}
for each vector field $\xi\in C^2(M\setminus\p M,NM)$. From \eqref{varHs}, $M$ is strictly stable if and only if the first eigenvalue of $L_M$ is positive.
Namely, there exists a constant $\ep_M>0$ such that
\begin{equation}\aligned\label{epMeigenvalueM}
\ep_M\int_M|\xi|^2\le-\int_M\lan \xi,L_M\xi\ran
\endaligned
\end{equation}
for any $\xi\in C_0^\infty(M\setminus\p M,NM)$.
\begin{lemma}\label{LemUNIGRAPH}
Let $M$ be an $n$-dimensional smooth strictly stable compact minimal submanifold with smooth boundary $\p M$ in $\R^{n+m}$.
There exists a constant $\de_M>0$ such that for any smooth minimal submanifold $S$ with boundary $\p M$, if $S=\{p+\Psi(p)|\, p\in M\}$ for some vector field $\Psi\in C_0^\infty(M,NM)$ with $|\na^\bot\Psi|_\Om+|(\na^\bot)^2\Psi|_\Om\le \de_M$, then $S=M$.
\end{lemma}
\begin{proof}
Let $\de=|\na^\bot\Psi|_\Om+|(\na^\bot)^2\Psi|_\Om$ with $\de\in(0,1]$, and $\Phi=\f1\de\Psi$. Then from \eqref{Qs}\eqref{Qs*}, one has
\begin{equation}\aligned\label{Hspabehnabe0}
L_M\Phi=\De^\bot_{M}\Phi+\sum_{i,j}\lan\Phi,A_{ij}\ran A_{ij}=-\de\mathcal{Q}_{\de,\Phi}
\endaligned
\end{equation}
with $|\mathcal{Q}_{\de,\Phi}|\le C_M'\left(|\Phi|+|\na^\bot_M\Phi|\right),$
where $C_M'$ is a constant depending only on $C_M$, and the diameter of $M$.
From \eqref{epMeigenvalueM}, we have
\begin{equation}\aligned\label{epMPhinaPhiM}
\ep_M\int_M|\Phi|^2\le&-\int_M\lan \Phi,L_M\Phi\ran=\de\int_M\lan \Phi,\mathcal{Q}_{\de,\Phi}\ran\\
\le& C_M'\de\int_M\left(|\Phi|^2+|\Phi|\cdot|\na^\bot_M\Phi|\right)
\le 2C_M'\de\int_M\left(|\Phi|^2+|\na^\bot_M\Phi|^2\right).
\endaligned
\end{equation}
Let $\{\textbf{n}^\a\}_{\a=1}^m$ be a fixed orthonormal frame for the normal space $NM$ such that $\na^\bot\mathbf{n}^\a=0$. Then we write $\Phi=\sum_\a\phi^\a\textbf{n}^\a$ for some vector-valued function $(\phi^1,\cdots,\phi^m)$ on $M$ with $\phi^\a=0$ on $\p M$. Denote $A_{ij}=h^\a_{ij}\textbf{n}^\a$.
From \eqref{Hspabehnabe0} (see also \eqref{Hspabehnabe}), one has
\begin{equation}\aligned
\De_M\phi^\a+\sum_{\a,i,j}h^\a_{ij}h^\be_{ij}\phi^\be +\de \mathcal{Q}^\a_{\de,\Phi}=0
\endaligned
\end{equation}
with $|\mathcal{Q}^\a_{\de,\Phi}|\le C_M'\left(|\Phi|+|\na^\bot_M\Phi|\right)$. From $W^{2,p}$-estimates,
\begin{equation}\aligned\label{phiW21L2}
\sum_\a||\phi||_{W^{1,2}(M)}\le C_M^*\sum_\a||\phi^\a||_{L^2(M)},
\endaligned
\end{equation}
where $C_M^*$ is a positive constant depending on the geometry of $M$.
Combining \eqref{epMPhinaPhiM} and \eqref{phiW21L2}, we deduce $\Phi\equiv0$ if we choose $\de$ sufficiently small depending on $M$. This completes the proof.
\end{proof}

Let $\Om$ be a smooth mean convex bounded domain in $\R^n$. Let $\varphi$ be a smooth function on $\p\Om$.
Let $\psi^1,\cdots,\psi^{m-1}$ be smooth functions on $\p\Om$ with $\sum_{\a=1}^{m-1}\sum_{i=0}^3|D^i\psi^\a|_{\p\Om}=1$.
Let $$\G_{t}=\{(x,t\psi^1(x),\cdots,t\psi^{m-1}(x),\varphi(x))\in\R^n\times\R^m|\ x\in\p\Om\}$$
for all $t\in[0,1]$.
\begin{proposition}\label{UNI}
There exists a constant $\de_{\Om,\varphi}>0$ depending only on $\Om$ and $\varphi$ such that for each $|t|\le\de_{\Om,\varphi}$
there is a unique minimal submanifold with boundary $\G_{t}$, and it coincides with the smooth solution obtained in Theorem \ref{ep-Ex}.
\end{proposition}
\begin{proof}
From Jenkins-Serrin \cite{JS}, there is a smooth solution $w$ to the minimal surface equation \eqref{MS0} with $w=\varphi$ on $\p\Om$. Let $M$ denote the graph over $\Om$ of the graphic function $(0,\cdots,0,\varphi)$ in $\R^{n+m}$. From \eqref{Lphiuuuu}-\eqref{LphiElliptic}, $M$ is a strictly stable minimal submanifold.\\
Let us prove this proposition by contradiction. Suppose that there are two sequences of smooth minimal submanifolds $\Si_k$ and $\Si_k'$ with $\Si_k\neq \Si_k'$ and $\p \Si_k=\p \Si_k'=\G_{t_k}$ for some sequence $t_k\rightarrow0$. By the Sobolev inequality on minimal submanifolds (see \cite{Si} for instance), the volumes of $\Si_k,\Si_k'$ are uniformly bounded. By compactness of varifolds (see \cite{LY,Si}), after choosing subsequences we may assume that $\Si_k$ converges to $\Si$ and $\Si_k'$ converges to $\Si'$ in the varifold sense, respectively.
Since $\p\Si=\p\Si'=\G_0$, by the maximum principle, $\Si$ and $\Si'$ live in an $(n+1)$-dimensional Euclidean space with boundary $=\{(x,,\varphi(x))\in\R^{n+1}|\ x\in\p\Om\}$.
Hence, $\Si=\Si'=M$. From Allard's regularity theorem \cite{A0} and the reflection principle (after flattening the boundary),  $\Si_k$ and $\Si_k'$ both  converge to $M$ smoothly since $\Si_k,\Si_k'$ are graphs over $\Om$.
In other words, there are smooth solutions $u_k,u'_k$ to the minimal surface system with $u_k=w_k=(t_k\psi^1,\cdots,t_k\psi^{m-1},\varphi)$ on $\p\Om$ such that $\Si_k=\mathrm{graph}_{u_k}$, $\Si'_k=\mathrm{graph}_{u_k'}$, and  $u_k$ and $u_k'$ both  converge smoothly to $(0,\cdots,0,w)$ on $\overline{\Om}$.
So, the strict stability of  $M$ implies that  $\Si_k$ and $\Si_k'$ both are strictly stable for sufficiently large $k>0$.
Moreover, for sufficiently large $k>0$, $S'_k$ can be seen as a graph over $S_{k}$ with $S'_k=\{p+\Psi_k(p)|\, p\in S_k\}$ for some vector field $\Psi_k\in C_0^\infty(S_k,NS_k)$ such that $\Psi_k$ converges to zero smoothly as $k\rightarrow\infty$. From Lemma \ref{LemUNIGRAPH}, we deduce $\Psi_k=0$ for sufficiently large $k>0$. This is a contradiction. We complete the proof.
\end{proof}


\section{Appendix I. Calculations for graphs over a submanifold}

In this appendix, we will calculate the mean curvature vectors for a one-parameter family $M_s$ over an $n$-dimensional smooth embedded submanifold $M$ in $\R^{n+m}$ (see the case of hypersurfaces by Colding-Minicozzi in \cite{CM1}). Let $\na^\bot$ denote the normal connection in $NM$ defined by
$$\na^\bot_X\nu=\left(\bn_X\nu\right)^\bot$$
for any $X\in\G(TM)$ and $\nu\in\G(NM)$.
Let $\{\textbf{n}^\a\}_{\a=1}^m$ be a fixed orthonormal frame for the normal space $NM$ such that $\na^\bot\mathbf{n}^\a=0$, and $(\phi^1,\cdots,\phi^m)$ be a vector-valued function on $M$. Let $M_s$ be given by
\begin{equation}\aligned\nonumber
F(\cdot,s):\ M\rightarrow\R^{n+m}\quad \mathrm{with}\ \ \ F(p,s)=p+s\phi^\a\textbf{n}^\a.
\endaligned
\end{equation}
Let $\{e_i\}_{i=1}^n$ be a local orthonormal frame for the tangent space $TM$. Let $A_{ij}=\overline{\na}_{e_i}e_j-\na_{e_i}e_j$ be the components of the second fundamental form of $M$, and
$$h^\a_{ij}=\lan A_{ij},\textbf{n}^\a\ran=\lan \overline{\na}_{e_i}e_j,\textbf{n}^\a\ran.$$
We extend both the functions $\phi^1,\cdots,\phi^m$ and the frame $\{e_i\}_{i=1}^n$ to a small neighborhood of $M$ by parallel translation along the normal frame, so that $\lan \textbf{n}^\a,\na\phi^\be\ran=0$ and $\bn_{\textbf{n}^\a}e_i=0$.

The tangent space of $M_s$ is spanned by $\{F_i\}_{i=1}^n$, where
\begin{equation}\aligned
F_i(p,s)=dF_{(p,s)}(e_i(p))=e_i(p)+s\phi^\a_i\textbf{n}^\a(p)-s\phi^\a h^\a_{ik}e_k(p).
\endaligned
\end{equation}
Note that $\{F_i\}$ is not orthonormal in general.
The metric $g_{ij}(p,s)$ of $M_s$ at $F(p,s)$ related to the frame $F_i$ is
\begin{equation}\aligned
g_{ij}(p,s)=\lan F_i,F_j\ran=\de_{ij}-2s\phi^\a h^\a_{ij}+s^2\phi^\a_i\phi^\a_j+s^2\phi^\a\phi^\be h^\a_{ik}h^\be_{jk}.
\endaligned
\end{equation}
Let $\mathbf{a}$ be the matrix with  elements $\mathbf{a}_{ij}=\de_{ij}-s\phi^\a h^\a_{ij}$, then $\mathbf{a}$ is positive definite if $s\phi$ is sufficiently small. Let $\mathbf{a}^{-1}$ be the inverse matrix of $\mathbf{a}$.
Put
\begin{equation}\aligned
\tilde{\textbf{n}}^\a_s(p)=\textbf{n}^\a(p)-s\na\phi^\a-s^2\left(\mathbf{a}^{-1}\right)_{ij}\phi^\be h^\be_{ik}\phi^\a_ke_j(p),
\endaligned
\end{equation}
then
$$\lan F_i(p,s),\tilde{\textbf{n}}^\a_s(p)\ran=s\phi^\a_i-s\phi^\a_j\mathbf{a}_{ij}-s^2\phi^\be h^\be_{ik}\phi^\a_k=0$$
for any $i=1,\cdots,n$ and $\a=1,\cdots,m$.

Let $P_{s,\phi}$ and $Q_{s,\phi}$ stand for general   functions of the form
\begin{equation}\aligned\label{APs}
P_{s,\phi}=f_0*\phi^\a*\phi^\be+f_1*\phi^\a*\na\phi^\be+f_2*\na\phi^\a*\na\phi^\be
\endaligned
\end{equation}
and
\begin{equation}\aligned\label{AQs1}
Q_{s,\phi}=&\tilde{f}_0*\phi^\a*\phi^\be+\tilde{f}_1*\phi^\a*\na\phi^\be+\tilde{f}_2*\na\phi^\a*\na\phi^\be\\
&+\tilde{f}_3*\phi^\a*\na^2\phi^\be+\tilde{f}_4*\na\phi^\a*\na^2\phi^\be,
\endaligned
\end{equation}
where $f_i$ are smooth vector fields depending on $s$, $\phi^\a,\na\phi^\a$, and
$\tilde{f}_i$ are smooth vector fields depending on $s$, $\phi^\a$, $\na\phi^\a$, $\na^2\phi^\a$ such that if $\sum_{0\le k\le2,\a}|\na^k\phi^\a|$ is bounded, then $\sum_{i=0}^2\left(|f_i|+|\na f_i|\right)+\sum_{i=0}^4|\tilde{f}_i|$ is bounded for the sufficiently small $s>0$. Note that the precise form of $P_{s,\phi}$ and $Q_{s,\phi}$ may be  different  even in the same line.

From now on, we assume that $s$ is sufficiently small. Then
\begin{equation}\aligned\label{AQs2}
\left|\tilde{\textbf{n}}^\a_s(p)\right|=\sqrt{1+s^2P_{s,\phi}}=1+s^2P_{s,\phi}.
\endaligned
\end{equation}
Let $$\textbf{n}^\a_s(p)=\f{\tilde{\textbf{n}}^\a_s(p)}{\left|\tilde{\textbf{n}}^\a_s(p)\right|},$$
then $\{\textbf{n}^\a_s(p)\}$ forms a basis (not necessarily orthonormal) for the normal space $NM_s$ at the point $F(p,s)$ and $\textbf{n}^\a_0(p)=\textbf{n}^\a(p)$.
Then one has
\begin{equation}\aligned\label{Ana}
\textbf{n}^\a_s(p)=\textbf{n}^\a(p)-s\na\phi^\a+s^2P_{s,\phi}.
\endaligned
\end{equation}
From $\bn_{\textbf{n}^\a}e_i=0$, a direct computation implies
\begin{equation}\aligned
\f{\p}{\p s}\left\lan\bn_{e_i}\textbf{n}^\a,e_j\right\ran=&\phi^\be\bn_{\textbf{n}^\be}\left\lan\bn_{e_i}\textbf{n}^\a,e_j\right\ran
=\phi^\be\left\lan\bn_{\textbf{n}^\be}\bn_{e_i}\textbf{n}^\a,e_j\right\ran\\
=&\phi^\be\left\lan\bn_{e_i}\bn_{\textbf{n}^\be}\textbf{n}^\a,e_j\right\ran+\phi^\be\left\lan\bn_{[\textbf{n}^\be,e_i]}\textbf{n}^\a,e_j\right\ran
\endaligned
\end{equation}
at the point $F(p,s)$.
Since
\begin{equation}\aligned
\left\lan\bn_{\textbf{n}^\be}\textbf{n}^\a,e_i\right\ran=-\left\lan\textbf{n}^\a,\bn_{\textbf{n}^\be}e_i\right\ran=0,
\endaligned
\end{equation}
then with $\na^\bot\mathbf{n}^\a=0$ one has
\begin{equation}\aligned
\left\lan\bn_{e_i}\bn_{\textbf{n}^\be}\textbf{n}^\a,e_j\right\ran=-\left\lan\bn_{\textbf{n}^\be}\textbf{n}^\a,\bn_{e_i}e_j\right\ran=0,
\endaligned
\end{equation}
and then
\begin{equation}\aligned\label{fppseiejna}
\f{\p}{\p s}\left\lan\bn_{e_i}\textbf{n}^\a,e_j\right\ran=&\phi^\be\left\lan\bn_{[\textbf{n}^\be,e_i]}\textbf{n}^\a,e_j\right\ran
=-\phi^\be\left\lan\bn_{\bn_{e_i}\textbf{n}^\be}\textbf{n}^\a,e_j\right\ran\\
=&-\phi^\be \left\lan\bn_{e_k}e_j,\textbf{n}^\a\right\ran\left\lan\bn_{e_i}e_k,\textbf{n}^\be\right\ran.
\endaligned
\end{equation}
Hence
\begin{equation}\aligned
\f{\p}{\p s}\Big|_{(p,0)}\bn_{e_i}\textbf{n}^\a=-\phi^\be h_{ik}^\be h_{jk}^\a e_j.
\endaligned
\end{equation}
Taking the derivative $\f{\p}{\p s}$ again on  both sides of \eqref{fppseiejna} implies
\begin{equation}\aligned
\bn_{e_i}\textbf{n}^\a=-h^\a_{ij}e_j(p)-s\phi^\be h_{ik}^\be h_{jk}^\a e_j(p)+s^2P_{s,\phi}.
\endaligned
\end{equation}
Denote $F'_i=F'_i(p,0)=\phi^\a_i\textbf{n}^\a(p)-\phi^\a h^\a_{ik}e_k(p)$, then
\begin{equation}\aligned
s\bn_{F'_i}\textbf{n}^\a=-s\phi^\be h_{ik}^\be\bn_{e_k}\textbf{n}^\a=s\phi^\be h_{ik}^\be h_{jk}^\a e_j(p)+s^2P_{s,\phi}.
\endaligned
\end{equation}
Hence at the point $F(p,s)$ we get
\begin{equation}\aligned
-\bn_{F_i}\textbf{n}^\a_s(p)=&-\bn_{e_i}\textbf{n}^\a+s\na_{e_i}\na\phi^\a-s\bn_{F'_i}\textbf{n}^\a+s^2Q_{s,\phi}\\
=&h^\a_{ij}e_j(p)+s\mathrm{Hess}_{\phi^\a}\left(e_i(p),e_j(p)\right)e_j(p)+s^2Q_{s,\phi}.
\endaligned
\end{equation}
Therefore, at the point $F(p,s)$ one has
\begin{equation}\aligned
\left\lan\na_{F_i}F_j,\textbf{n}^\a_s(p)\right\ran=-\left\lan\bn_{F_i}\textbf{n}^\a_s(p),F_j\right\ran=h^\a_{ij}+s\mathrm{Hess}_{\phi^\a}\left(e_i(p),e_j(p)\right)-s\phi^\be h^\be_{jk}h^\a_{ik}+s^2Q_{s,\phi}.
\endaligned
\end{equation}
Let $H_s(p)=(H^1_s(p),\cdots,H^m_s(p))$ denote the mean curvature vector of $M_s$ at $F(p,s)$.
Since $g^{ij}=\de_{ij}+2s\phi^\a h^\a_{ij}+s^2P_{s,\phi}$, then
\begin{equation}\aligned
H^\a_s(p)\triangleq&g^{ij}\left\lan\bn_{F_i}F_j,\textbf{n}^\a_s(p)\right\ran\\
=&H^\a(p)+s\left(\De_M\phi^\a+\phi^\be h^\be_{ij}h^\a_{ij}\right)+s^2Q_{s,\phi}.
\endaligned
\end{equation}
Let
$$\tau_{\a\be}(p,s)=\lan\textbf{n}^\a_s(p),\textbf{n}^\be_s(p)\ran=\de_{\a\be}+s^2P_{s,\phi}$$
at $F(p,s)$, and $(\tau^{\a\be}(p,s))$ be the inverse matrix of $(\tau_{\a\be}(p,s))$.
Then
\begin{equation}\aligned\label{Hspabehnabe}
H_s(p)=\tau^{\a\be}H^\a_s(p)\textbf{n}^\be_s=H^\a(p)\textbf{n}^\a_s(p)+s\left(\De_M\phi^\a+\phi^\be h^\be_{ij}h^\a_{ij}\right)\textbf{n}_s^\a(p)+s^2Q_{s,\phi}.
\endaligned
\end{equation}
Put $\De^\bot_M \xi=(\na^\bot)^2\xi(e_i,e_i)$ for  $\xi\in\G(NM)$, and $\Phi=\phi^\a \textbf{n}^\a$. In particular, if $M$ is a minimal submanifold, then
\begin{equation}\aligned
H_s(p)=s\left(\De^\bot_M\Phi+\lan\Phi, A_{ij}\ran A_{ij}\right)+s^2Q_{s,\phi}.
\endaligned
\end{equation}

\section{Appendix II: Algebraic inequalities}

Here we state an algebraic result, which is sharp for $m=1$.
\begin{lemma}\label{smallT}
Let $S_{i\a}$ be an $(n\times m)$-real matrix. Put $S_\a=\left(\sum_{i=1}^n S_{i\a}^2\right)^{1/2}$ for each $\a=1,\cdots,m$. If $\sum_{\a=1}^{m-1}S_\a\le\min\{2\sqrt{2},2/S_m\}$, then
\begin{equation}\aligned\label{Saj1*}
\mathrm{det}\left(\de_{ij}+\sum_{\a=1}^m S_{i\a}S_{j\a}\right)\le1+\left(\sum_{\a=1}^mS_\a\right)^2.
\endaligned
\end{equation}
\end{lemma}
\begin{proof}
Let us prove this lemma by induction. Clearly, the inequality \eqref{Saj1*} holds for $m=1$.
Assume that \eqref{Saj1*} holds for $m-1$ with $m\ge2$.
Let $\La$ be a diagonal $(n\times n)$-matrix with eigenvalues $1,\cdots,1,\sqrt{1+S_m^2}$.
Then there is an orthonormal matrix $P=(p_{ij})_{n\times n}$ such that the matrix $(\de_{ij}+S_{im}S_{jm})=P\La^2P^T$.
Let $S'_{i\a}=\sum_{k=1}^np_{ki}S_{k\a}$ for $i=1,\cdots,n-1$, and $S'_{n\a}=(1+S_m^2)^{-1/2}\sum_{k=1}^np_{kn}S_{k\a}$.
Let $Q=(q_{ij})$ be a matrix defined by $q_{ij}=\sum_{\a=1}^{m-1} S_{i\a}S_{j\a}$, and $Q'=(q'_{ij})$ be a matrix defined by $q'_{ij}=\sum_{\a=1}^{m-1} S'_{i\a}S'_{j\a}$.
Then $Q'=\La^{-1}P^TQP\La^{-1}$, and
\begin{equation}\aligned\label{detdeijSia}
\mathrm{det}\left(\de_{ij}+\sum_{\a=1}^m S^\a_iS^\a_j\right)=&\left|\mathrm{det}(\La P^T)\right|^2\mathrm{det}\left(\La^{-1}P^T(P\La^2P^T+Q)P\La^{-1}\right)\\
=&\left(1+S_m^2\right)\mathrm{det}\left(I+Q'\right).
\endaligned
\end{equation}
Moreover, let $S'_\a=\left(\sum_{i=1}^n (S'_{i\a})^2\right)^{1/2}$ for each $\a=1,\cdots,m-1$, then
\begin{equation}\aligned
(S'_\a)^2=\sum_{i=1}^n (S'_{i\a})^2\le \sum_{i=1}^n \left(\sum_{k=1}^np_{ki}S_{k\a}\right)^2=\sum_{i=1}^n \sum_{k,l=1}^np_{ki}p_{li}S_{k\a}S_{l\a}=\sum_{k=1}^n S_{k\a}^2=S_\a^2,
\endaligned
\end{equation}
Since $\sum_{\a=1}^{m-1}S_\a\le\min\{2\sqrt{2},2/S_m\}$, then for $m\ge3$ we have $\sum_{\a=1}^{m-2}S'_\a\le\sum_{\a=1}^{m-2}S_\a\le 2\sqrt{2}$ and
\begin{equation}\aligned
\sum_{\a=1}^{m-2}S'_\a S'_{m-1}\le\f14\left(\sum_{\a=1}^{m-2}S'_\a+S'_{m-1}\right)^2\le\f14\left(\sum_{\a=1}^{m-1}S_\a\right)^2\le2.
\endaligned
\end{equation}
By assumption, \eqref{Saj1*} holds for $m-1$ with $m\ge2$, which implies
\begin{equation}\aligned\label{IQ'}
\mathrm{det}\left(I+Q'\right)\le 1+\left(\sum_{\a=1}^{m-1}S_\a\right)^2
\endaligned
\end{equation}
for $m\ge3$.
It is clear that \eqref{IQ'} holds for $m=2$. Substituting \eqref{IQ'} into \eqref{detdeijSia} implies
\begin{equation}\aligned
&\mathrm{det}\left(\de_{ij}+\sum_{\a=1}^m S_{i\a}S_{j\a}\right)\le\left(1+S_m^2\right)\left(1+\left(\sum_{\a=1}^{m-1}S_\a\right)^2\right)\\
=&1+\left(\sum_{\a=1}^{m-1}S_\a\right)^2+S_m^2+\left(\sum_{\a=1}^{m-1}S_\a\right)^2S_m^2\\
\le& 1+\left(\sum_{\a=1}^{m-1}S_\a\right)^2+S_m^2+2\left(\sum_{\a=1}^{m-1}S_\a\right)S_m=1+\left(\sum_{\a=1}^{m}S_\a\right)^2,
\endaligned
\end{equation}
where we have used $\sum_{\a=1}^{m-1}S_\a S_m\le2$ in the above inequality.
This completes the proof.
\end{proof}

If we assume $\sum_{\a=1}^mS_\a\le2\sqrt{2}$, then from the Cauchy-Schwarz inequality,
$$\sum_{\a=1}^{m-1}S_\a S_m\le\f14\left(\sum_{\a=1}^{m-1}S_\a+S_m\right)^2\le\f14\times8=2.$$
From Lemma \ref{smallT}, we immediately have the following result.
\begin{lemma}\label{sssss}
Let $S_{i\a}$ be an $(n\times m)$-real matrix. Put $S_\a=\left(\sum_{i=1}^n S_{i\a}^2\right)^{1/2}$ for  $\a=1,\cdots,m$. If $\sum_{\a=1}^mS_\a\le2\sqrt{2}$, then
\begin{equation}\aligned\label{le9*}
\mathrm{det}\left(\de_{ij}+\sum_{\a=1}^m S_{i\a}S_{j\a}\right)\le1+\left(\sum_{\a=1}^mS_\a\right)^2.
\endaligned
\end{equation}
\end{lemma}

\bibliographystyle{amsplain}

\end{document}